\documentclass[a4paper, 11pt, twoside, notitlepage]{amsart}

\usepackage{amsmath,amscd}
\usepackage{amssymb}
\usepackage{amsthm}
\usepackage{comment}
\usepackage{graphicx, xcolor}
\usepackage{subcaption}
\usepackage{mathrsfs}
\usepackage[ocgcolorlinks, linkcolor=blue]{hyperref}
\usepackage[normalem]{ulem}
\usepackage{xcolor,cancel}

\newcommand\redsout{\bgroup\markoverwith{\textcolor{red}{\rule[0.5ex]{2pt}{0.4pt}}}\ULon}

\usepackage{bm}
\usepackage{bbm}
\usepackage{url}

\usepackage[utf8]{inputenc}
\usepackage{mathtools,amssymb}
\usepackage{esint}
\usepackage{tikz}
\usepackage{dsfont}
\usepackage{relsize}
\usepackage{url}
\usepackage{xcolor}
\usepackage{graphicx}
\usepackage{mathrsfs}
\usepackage[shortlabels]{enumitem}
\usepackage{lineno}
\usepackage{amsmath}
\usepackage{enumitem}
\usepackage{amsthm} 
\usepackage{verbatim}
\usepackage{dsfont}
\numberwithin{equation}{section}

\allowdisplaybreaks
 
\def\embed{\hookrightarrow}
\newsavebox{\mycases}

\mathtoolsset{showonlyrefs}

\graphicspath{{images/}}

\newtheorem{theorem}{Theorem}[section]
\newtheorem{lemma}[theorem]{Lemma}
\newtheorem{definition}[theorem]{Definition}
\newtheorem{corollary}[theorem]{Corollary}

\newtheorem{proposition}[theorem]{Proposition}

\newtheorem{remark}[theorem]{Remark}
\newtheorem{assumption}[theorem]{Assumption}
\newtheorem{example}[theorem]{Example}

\def\qda{q^\delta_{h,\alpha}}

\title[]{Finite element error analysis for elliptic parameter identification with power-type nonlinearity}

\author[D.-H. Chen]{De-Han Chen}
\address{School of Mathematics and Statistics, Central China Normal University. Wuhan, China}
\email{dhchen@ccnu.edu.cn}

\author[Y.-H. Lin]{Yi-Hsuan Lin}
\address{Department of Applied Mathematics, National Yang Ming Chiao Tung University, Hsinchu, Taiwan}
\email{yihsuanlin3@gmail.com}

\author[I. Yousept]{Irwin Yousept$^{\dagger}$}
\address{Universit\"at Duisburg-Essen, Essen, Germany}
\email{irwin.yousept@uni-due.de}

\thanks{$^{\dagger}$Corresponding author.}

\keywords{numerical parameter reconstruction, elliptic equations with power-type nonlinearity, finite element approximations, conditional stability estimates, a priori error estimates}
\subjclass[2020]{}

\newcommand{\R}{{\mathbb R}}

\newcommand{\N}{{\mathbb N}}

\newcommand{\eps}{\epsilon}

\newcommand {\p} {\partial}

\newcommand{\LC}{\left(}
\newcommand{\RC}{\right)}
\newcommand{\wt}{\widetilde}

\newcommand{\norm}[1]{\lVert #1 \rVert}
\newcommand{\abs}[1]{\left\lvert #1 \right\rvert}
\DeclareMathOperator{\dist}{dist} 

\begin{document}

	\maketitle
	
	\begin{abstract}
This paper studies the numerical analysis of a parameter identification problem governed by elliptic equations with power-type nonlinearity. We propose a numerical reconstruction via a suitable least-squares minimization problem based on piecewise linear finite elements. As one of our main novelties, we establish conditional stability estimates at the continuous level, which form the theoretical foundation of the present finite element analysis. Our stability analysis relies on tailored analytical tools, including Hardy-type inequalities, fractional Gagliardo–Nirenberg inequalities, and weighted spaces with singular distance weights. By invoking the achieved conditional stability together with the Carstensen quasi-interpolation operator and associated estimates in negative Sobolev spaces, we derive a priori error estimates for the proposed finite element approximation in terms of the mesh size, the regularization parameter, the noise level, and the nonlinearity exponent. Our results extend the recent stability and error estimates for the linear case by Jin et al. \cite{jin2022convergence} and sharpen their error estimates and convergence order under weaker regularity assumptions.

	\end{abstract}
	

	\section{Introduction}\label{sec: introduction}
    
	This paper examines the error analysis of a finite element   method (FEM) for a parameter identification problem governed by an elliptic partial differential equation (PDE) involving power-type nonlinearity as follows:
	\begin{equation}\label{eq: main}
		\begin{cases}
			-\nabla \cdot (\sigma\nabla u)+ qu^m =f &\text{ in }\Omega \\  
			u  =0 &\text{ on }\p \Omega
		\end{cases}
	\end{equation}
	for a bounded Lipschitz domain $\Omega\subset \R^n$ with $n \ge 2$ and an odd natural number $m \ge 1$. Under this nonlinear PDE-model, we focus on the ill-posed inverse problem for the reconstruction of the zero-order coefficient $q$ in \eqref{eq: main}. To be more precise, given $\sigma$ and $f$ satisfying certain regularity conditions (see Assumption \ref{Assump: regularity}) and noisy measurement  data $y^\delta$ for the unknown true weak solution $u^\dag$ of \eqref{eq: main}, we numerically reconstruct   the unknown true coefficient $q^\dagger$ based on piecewise linear FEM, by solving the following  least squares minimization problem:
    \begin{equation}\label{minimzation intro}
    \begin{aligned}
        &\textrm{Minimize } \frac{1}{2}\big\|u_h(q_h)-y^\delta\big\|^2_{L^2(\Omega)}+\frac{\alpha}{2}\big(\|q_h\|_{L^2(\Omega)}^2+\|\nabla q_h\|^2_{L^2(\Omega)}\big) \\
	&\textrm{subject to } q_h \in S_h \textrm{ and } \underline q \le q_h  \le \overline q. 
    \end{aligned} 
       \end{equation}
In the setting of \eqref{minimzation intro}, $ S_h \subset H^1(\Omega)$ stands for the standard piecewise linear finite element space  (see \eqref{FEM}), $u_h(q_h)$ the associated finite element approximation of the forward problem (see \eqref{numerical:op} for its precise definition), $\alpha>0$ the Tikhonov regularization parameter, and $0 \le \underline q \le \overline q$ some prescribed lower and upper bounds for the feasible zero-order coefficient. Our goal is to establish rigorous error estimates for the reconstruction error both in the coefficient and in the corresponding weak PDE-solution with respect to the discretization parameter $h$, the regularization parameter $\alpha$,  the noisy level $\delta$, as well as the nonlinearity exponent $m$.

Existing contributions to such a parameter identification problem are predominantly devoted to the linear forward equations  (see \cite{CJZ20, EKN1989_convergence, EZ2000_new, HQ2010_convergence, KJ2011_new, JFZ2012}), focusing primarily on the convergence rate of the Tikhonov regularization at the continuous level.  To obtain a computationally tractable formulation,  as we propose in \eqref{minimzation intro}, the regularized output least-squares minimization problem is discretized using  FEM due to its adaptability to complex geometries and low regularity. The finite element discretization, however,  leads to additional numerical errors in the reconstruction, making a rigorous a priori error analysis essential for the accuracy of the computed approximation.  The plain finite element convergence, without error estimates,  has been widely studied,  see, e.g., \cite{mchen99,bangti10} for elliptic inverse coefficient problems, \cite{bangti12,keu98,xie05} for
 parabolic inverse coefficient problems, and \cite{feng12} for Maxwell inverse coefficient problems. However, results regarding a priori error estimates for the reconstruction error remain limited. 
 
Wang and Zou \cite{wang10} were the first to initiate the finite element error analysis and derive weighted-norm estimates for elliptic and parabolic inverse conductivity problems. Their analysis, however, requires full-time measurements and offers no standard estimates in $L^2$-norm. Jin and Zhou \cite{bangti21} extended the developed results \cite{wang10} to a more general setting, allowing measurements on a subinterval $[T-\sigma, T]$ and reducing the underlying regularity assumptions. More recently, Jin et al. \cite{jin2022convergence} rigorously analyzed the finite element reconstruction of $q$ in \eqref{eq: main} for the linear case $m=1$ and its linear parabolic counterpart. Their major contributions include conditional stability estimates at the continuous level and finite element error estimates with respect to the noisy level $\delta$, the regularization parameter $\alpha$, and the mesh size $h$. In particular, they introduced a novel tailored test function, different from \cite{wang10}
and \cite{bangti21},   serving as the theoretical basis for deriving    error estimates in the standard $L^2$-norm and the interior
$L^2$-norm for the finite element solutions.

To the best of our knowledge, this paper is the first to address and explore finite element estimates in the nonlinear case \eqref{eq: main}. Although the linear case $m=1$  has readily been investigated in \cite{jin2022convergence}, their results, however, cannot be transferred to the nonlinear case $m >1$ and require substantial extension.  Our main contributions are threefold:  First, we establish novel conditional stability estimates for \eqref{eq: main}  (see Theorem \ref{thm: cse})  using a functional-analytic framework based on weighted Sobolev spaces with singular weights, Hardy-type inequalities, and fractional Gagliardo-Nirenberg inequalities. In particular, our analysis avoids the explicit use of a nonstandard test-function required in the previous contribution \cite{jin2022convergence} for the linear case $m=1$.   Second, within this functional framework, we not only obtain stronger stability estimates than those available so far only for the linear case   \cite{jin2022convergence}, but also manage to verify the required positivity condition   (see Assumption \ref{Assump: cse}) for the nonlinear case $m>1$ (see Proposition \ref{prop: D1}).
Third, by combining the improved conditional stability with the Carstensen quasi-interpolation operator \cite{carstensen1999quasi} and suitable estimates in negative Sobolev spaces, we derive in Theorem \ref{final theo}  a priori error estimates in terms of the discretization parameter $h$, the regularization parameter \( \alpha \), the noise level \( \delta \), and the nonlinearity exponent \( m \). In particular, for the linear case $m=1$,  our framework achieves sharper error estimates than \cite{CJZ20} under the weaker regularity condition $q^\dag \in H^1(\Omega)$ in contrast to the $H^2$-regularity assumption used previously in \cite{jin2022convergence}. This relaxation significantly broadens the applicability of the proposed approach (see Remark \ref{remark:comp:FEM}).\\

	\noindent \textbf{Organization of this article.} In Section~\ref{sec: forward problem}, we review relevant function spaces as well as the well-posedness and regularity results for the semilinear elliptic equation~\eqref{eq: main}. Section~\ref{sec: convergence rates} is devoted to the establishment of the conditional stability results under appropriate assumptions. In Section~\ref{sec: finite element}, we provide the rigorous error estimates for the finite element approximation. The final section is devoted to numerical tests that illustrate the theoretical findings achieved.

	\section{Preliminaries}\label{sec: forward problem}
	Let us begin by introducing our notation. If $V$ is a normed vector space,  then $\|\cdot\|_{V}: V \to \R_+\cup \{0\}$ stands for the norm used in the space $V$, and $\langle \cdot,\cdot\rangle_{V^*,V}$  denotes  the duality pairing with the corresponding dual space $V^*$. If $V$ is a Hilbert space, then $(\cdot,\cdot)_V$ and $\|\cdot\|_{V}$ denote, respectively, its scalar product and the induced norm. For the special case $V = \R^n$,  we use the Euclidean scalar product and the Euclidean norm that are denoted by a dot and by $|\cdot|:= \|\cdot\|_{\R^n}$.   
    
    Now, for given $s\in [0,\infty)$ and  $p\in [1,\infty]$, we denote by 
	$W^{s,p}(\Omega)$ the conventional Sobolev space  (see, e.g., \cite[Definition 8.10.10]{bhattacharyya2012distributions}).  The conjugate exponent of $p$ is denoted by $p'$, i.e., $\frac{1}{p} +\frac{1}{p'}=1$.   We recall that, in the case of $1\leq p < \infty$, the Sobolev norm $\|\cdot\|_{{W}^{s,p}(\Omega)}$ is given by
	\begin{equation}\label{Wsp:norm}
		\|u\|_{{W}^{s,p}(\Omega)} := \begin{cases}
			\big( \sum_{|\beta|\leq s}\|D^\beta u\|_{L^p(\Omega)}^p\big)^{\frac{1}{p}} \quad  \text{ if }\  s\in \mathbb N\\ 
			\big( \sum_{|\beta|\leq m}\|D^\beta u\|_{L^p(\Omega)}^p  +\sum_{|\beta|= m}         \int_{\Omega}\int_\Omega\frac{|D^\beta u(x)-D^\beta u(y)|^p}{|x-y|^{n+\sigma p}}\, dxdy\big)^{\frac{1}{p}} \\
			\qquad \qquad \qquad \qquad \qquad \text{ if }\  s= m+\sigma \text{ for }m \in \mathbb N \text{ and } \sigma\in (0,1)
		\end{cases}   
	\end{equation}
	where $D^\beta u$ denoting the  weak partial derivative of $u$ with multi-index $\beta\in (\N \cup \{0\})^n$. Similarly,   $W^{s,p}(\Omega)^{n\times n}$ denotes  the Sobolev space of $\R^{n\times n}$-valued functions. 
	Furthermore, 
	$C_0^\infty(\Omega)$ stands for the set of all smooth functions $\varphi$ with a compact support $\textup{supp}\,\varphi$ contained in $\Omega$. By $W_0^{s,p}(\Omega)$, we denote the closure of $C_0^\infty(\Omega)$ with respect to the norm $\|\cdot\|_{W^{s,p}(\Omega)}$. Finally,  ${W}^{-s,p}(\Omega)$ stands for the dual space of  $W_0^{s,p'}(\Omega)$.

	\subsection{Well-posedness}
	Let us discuss the well-posedness for the forward problem \eqref{eq: main} as well as the associated maximum and comparison principles. We begin by formulating the standing assumption for \eqref{eq: main}:
	\begin{assumption}\label{Assump: regularity} Let $\Omega \subset \R^n$ with $n \ge 2$ be a bounded Lipschitz domain, $m=2l+1$ with $l\in \N\cup\{0\}$, $r > n/2$, $f \in L^r(\Omega)$,   and  $\sigma=\LC\sigma_{ij}(x)\RC _{1\leq i,j\leq n}\in L^\infty(\Omega)^{n\times n}$ satisfying the following symmetric and elliptic conditions:   
		\begin{equation}\label{ellipticity}
			\begin{cases}
				\underline{\lambda}\,|\xi|^2 \le \sigma(x) \xi \cdot \xi  \le  \overline{\lambda}|\xi|^2   &\text{ for  all }\xi \in \R^n \textrm{ and a.e. } x \in \Omega\\
				\sigma_{ij}(x)=\sigma_{ji}(x) &\text{ for all }i,j=1,\ldots, n
			\end{cases}
		\end{equation} 
		for some fixed constants $0 < \underline \lambda \le \overline \lambda <\infty$, independent of $\xi$ and $x$. 
	\end{assumption}

	\begin{lemma}[{\cite[Theorem 1]{Meyer63_Lp} and \cite[Theorem 0.5]{JK95_Dirichlet}}]\label{lemma: Meyer-Jerison-Kenig}
		Let Assumption \ref{Assump: regularity}  be satisfied. Then, there exists a real number $\hat p>2$ such that for every $p \in (\hat p',\hat p)$ the linear operator  
		\begin{equation}
			\begin{split}
				-\nabla\cdot\sigma \nabla  : W^{1,p}_0(\Omega) \to W^{-1,p}(\Omega)
			\end{split}
		\end{equation}
		is a topological isomorphism. 
	\end{lemma}
	
	\begin{lemma}[Well-posedness]\label{thm: well-posed N}
		Let Assumption \ref{Assump: regularity}  hold and   $q \in L^r(\Omega)$ be nonnegative. Then,    the forward problem \eqref{eq: main} admits a unique weak solution $u\in H^1_0(\Omega)\cap C(\overline \Omega)$ satisfying
		\begin{equation}\label{est: well 0}
			\|u\|_{H^1(\Omega)}+ \norm{u}_{C(\overline\Omega)} \leq C \norm{f}_{L^r(\Omega)}
		\end{equation}
		with a constant $C>0$  independent of $u$, $f$, and $q$. Furthermore, there exists an $\overline p \in (2,\infty]$ such that the weak solution satisfies the regularity  $u \in W^{1,\overline p}(\Omega)$   and  
		\begin{equation}\label{est: well 1}
			\norm{u}_{W^{1,\overline p}(\Omega)}\leq C \big(\norm{f}_{L^r(\Omega)}+ \|q\|_{L^r(\Omega)}\norm{f}_{L^r(\Omega)}^m \big)
		\end{equation} 
		with a constant $C>0$  independent of $u$, $f$, and $q$. Moreover, if $\Omega$ is convex, $r>n$, and   $\sigma  \in W^{1,\omega}(\Omega)^{n\times n}$ for some $\omega>n$, then the proceeding regularity result and \eqref{est: well 1} hold with $\overline p=\infty$.   If additionally $\Omega$ is   convex or of class $C^{1,1}$ and $\sigma  \in W^{1,\infty}(\Omega)^{n\times n}$, then the weak solution satisfies the additional regularity  $u \in H^2(\Omega)$   and     
		\begin{equation}\label{est: well 2}
			\begin{split}
			    \|u(q)\|_{H^2(\Omega)}\leq C \big(\|f\|_{L^2(\Omega)}+\|q\|_{L^2(\Omega)}\norm{f}_{L^r(\Omega)}^m\big) 
			\end{split}
		\end{equation} 
		with a constant $C>0$  independent of $u$, $f$, and $q$.

	\end{lemma}

	\begin{proof}
		Thanks to Assumption \ref{Assump: regularity}, since $q \in L^r(\Omega)$ is a nonnegative function, the existence of a unique weak solution $u \in H^1(\Omega) \cap C(\overline \Omega)$ of \eqref{eq: main} satisfying \eqref{est: well 0} is a well-known result (see, e.g., \cite[Theorem 4.10]{TF_optimal_10}).  By definition, the unique weak solution $u \in H^1_0(\Omega) \cap C(\overline \Omega)$ of \eqref{eq: main} satisfies
		\begin{equation}\label{identity:elliptic} 
			    -\nabla \cdot (\sigma \nabla u)  =  f- qu^m \quad\text{ in } \Omega \quad \textup{and}
                \quad
				u=0 \quad\text{ on } \partial \Omega,
		\end{equation}
		in the weak sense. Since $u \in C(\overline \Omega)$ and $f,q \in L^r(\Omega)$, the Sobolev embedding theorem implies that the right-hand side in \eqref{identity:elliptic} enjoys the regularity property   
		\begin{equation}
			f - qu^m  \in     L^{r}(\Omega) \embed W^{1,\overline p'}_0(\Omega)^* \quad \textrm{with} \quad \overline p= 
			\begin{cases}
				\frac{rn}{n-r} &\textup{if } 
				r<n \\
				\infty &\textup{if }  r\geq n.
			\end{cases}
		\end{equation}
		Direct computation shows that  $\overline p>2$ holds since $r>\frac{n}{2}$. If necessary we may reduce $\overline p>2$ such that $\overline p\in (2,\hat p)$ with $\hat p \in (2,\infty)$ as in Lemma \ref{lemma: Meyer-Jerison-Kenig}. Thus, applying Lemma \ref{lemma: Meyer-Jerison-Kenig} to \eqref{identity:elliptic} implies   $u \in W^{1,\overline p}(\Omega)$ and
		\begin{equation}
			\begin{split}
				\|u\|_{W^{1,\overline p}(\Omega)}&\le C \|f - qu^m \|_{W^{1,\overline p'}_0(\Omega)^* } \le C \|f - qu^m \|_{L^r(\Omega) }  \\
				&\le C( \|f\|_{L^r(\Omega)}  +\|q\|_{L^r(\Omega)}\|u \|^m_{C(\overline{\Omega})}) 
				\underbrace{\le}_{\eqref{est: well 0}} C\big(\|f\|_{L^r(\Omega)}+ \|q\|_{L^r(\Omega)}\|f\|^m_{L^r(\Omega)}\big)
			\end{split}
		\end{equation}
		with a constant $C>0$  independent of $u$, $f$, and $q$.

		Now, suppose  additionally that $\Omega$ is convex, $r>n$, and $\sigma  \in W^{1,\omega}(\Omega)^{n\times n}$ for some $\omega>n$. Then,  thanks to \cite[Lemma 2.1]{li2017maximal},  the unique weak solution of \eqref{eq: main} fulfills
		\begin{equation}\label{est:grad}
			\begin{split}
			    \|\nabla u\|_{L^\infty(\Omega)}\leq C\|\nabla \cdot (\sigma \nabla u)\|_{L^r(\Omega)} 
			\underbrace{=}_{\eqref{identity:elliptic}} C \| f-  qu^m \|_{L^r(\Omega)}
			\underbrace{\le}_{\eqref{est: well 0}} C\big(\|f\|_{L^r(\Omega)} + \|q\|_{L^r(\Omega)}\|f\|^m_{L^r(\Omega)}\big)
			\end{split}
		\end{equation}
		with a constant $C>0$ independent of $u$, $f$, and $q$. Combining \eqref{est:grad} and  \eqref{est: well 0}, the desired estimate  \eqref{est: well 1} is obtained for $\overline p=\infty$. Analogously, if $\Omega$ is additionally convex or of class $C^{1,1}$ and $\sigma  \in W^{1,\infty}(\Omega)^{n\times n}$,   the classical $H^2$-regularity results \cite[Theorem 3.2.1.2]{Grisvard2011} and \cite[Corollary 2.2.2.4]{Grisvard2011} imply   $u \in H^2(\Omega)$ and the corresponding a priori estimate.
	\end{proof}

	For the upcoming results, we shall make use of the algebraic identity
    \begin{equation}\label{difference binomial}
			a^m -b^m = \Big(\sum_{k=0}^{m-1} a^kb^{m-1-k} \Big) (a-b) \quad \forall a, b \in \R.
		\end{equation}
    By the mean value theorem, we have  $ a^m-b^m = m (a+\theta_{ab}(b-a))^{m-1} (a-b)$ for some $\theta_{ab}\in (0,1)$, and applying this to \eqref{difference binomial}, it follows that
    \begin{equation}\label{difference binomial ineq}
      \sum_{k=0}^{m-1} a^kb^{m-1-k}   \geq 0  \quad \forall a,b \in \R,
	\end{equation}
    since $m$ is an odd natural number.

	\begin{lemma}[Maximum principle and comparison principle]\label{lemma:MP}
		Let Assumption \ref{Assump: regularity} hold,   $q \in L^r(\Omega)$ be nonnegative, and $u \in H^1_0(\Omega) \cap C(\overline \Omega)$ denote the unique weak solution to \eqref{eq: main}. Then, the following claims hold:
		\begin{enumerate}[(i)]
			\item\label{item 1 MP} If $f(x)\geq 0$ holds for a.e. $x \in \Omega$, then $u(x)\geq 0$ holds for all $x \in \overline \Omega$.
			
			\item \label{item 2 MP} If $f(x)>0$  holds for a.e. $x \in \Omega$, then $u(x)>0$ holds for all $x \in  \Omega$.
			
			\item\label{item 3 MP} Let $f_1,f_2 \in L^r(\Omega)$ and 
			$u_1,u_2\in H^1_0(\Omega) \cap C(\overline \Omega)$ denote the unique weak solution to \eqref{eq: main} with $f=f_1$ and $f=f_2$, respectively.  If $f_1(x)\geq f_2(x)$ holds for a.e. $x \in\Omega$, then   $u_1(x)\geq u_2(x)$ holds for all $x \in \overline \Omega$. 
		\end{enumerate}
		
	\end{lemma}
	
	\begin{proof}
		\ref{item 1 MP} Let us write $u=u^+-u^-$ where $u^+ :=\max\{u, 0\}\geq 0$ and $u^- :=\max\{ - u , 0\}\geq 0$.
		As $u\in H^1_0(\Omega)\cap C(\overline \Omega)$, well-known arguments (for example, see \cite[Chapter II, Theorem A.1]{KS2000_variational}) yield that $u^+, u^- \in H^1_0(\Omega)\cap C(\overline \Omega)$. Now, suppose that $f(x)\geq 0$ holds for a.e. $x \in \Omega$.  To show that $u(x)\geq 0$ for all $x \in \overline \Omega$, we consider $u^-\in H^1_0(\Omega) \cap C(\overline \Omega)$ as the test function in the weak formulation of \eqref{eq: main} to deduce that 
		\begin{equation}
			\begin{split}
				0 & \underbrace{\leq}_{\text{$f\geq 0$ a.e. }} \int_{\Omega} f u^- \, dx  
				=\int_{\Omega}  \sigma\nabla u  \cdot \nabla u^- \, dx + qu^m u^- \, dx \\
				&\ \ \ \  =\int_{\Omega} \sigma \nabla \LC u^+-u^-\RC \cdot \nabla u^- \, dx +\int_{\Omega} q\LC u^+-u^-\RC^m u^- \, dx\\
				& \,\underbrace{=}_{u^+\cdot u^- =0 } -\int_{\Omega}\sigma \nabla u^- \cdot \nabla u^- \, dx -\int_{\Omega}  q\LC u^-\RC^{m+1} \, dx 
				\underbrace{\le}_{u^-,q \ge 0} 0,  
			\end{split}
		\end{equation}
		from which it follows that $\|\nabla u^-\|_{L^2(\Omega)}^2 \le 0 $ implies $ u^- \equiv 0$.
		Thus, the assertion \ref{item 1 MP} follows.
		
		\ref{item 2 MP} Suppose that $f(x)>0$  holds for a.e. $x \in \Omega$. We proceed by contradiction and assume that there exists a nonempty open set $O \subset \Omega$ such that $u(x) = 0$ for all $x \in O$.
		Then, the weak formulation gives 
		\begin{equation}
			\begin{split}
				\int_{O} f\varphi \, dx =\int_{O}\sigma \nabla u \cdot \nabla \varphi \, dx + \int_{O} q u^m \varphi \, dx =0 \quad \forall \varphi\in C^\infty_0(O).
			\end{split}
		\end{equation}
		Thus, it follows that $f=0$ a.e. in $O$, which is a contradiction. In conclusion, along with \ref{item 1 MP}, it follows that $u$ is positive.
		
		\ref{item 3 MP} In view of \eqref{difference binomial}, we see that $v:= u_1-u_2 \in  H^1_0(\Omega) \cap C(\overline \Omega)$ is the unique weak solution to 
		\begin{equation} \label{lincomparison}
			-\nabla \cdot \LC \sigma\nabla v \RC +  {q\Big(\sum_{k=0}^{m-1} u_1^k u_2^{m-1-k} \Big)}  v = f_1 -f_2 \geq 0 \text{ in }\Omega, \quad 
			v  =0 \text{ on }\p \Omega.
		\end{equation}
		Note that  \eqref{lincomparison} is a special case of \eqref{eq: main} with $m=1$ and the coefficient $q\Big(\sum_{k=0}^{m-1} u_1^k u_2^{m-1-k} \Big)\in L^r(\Omega)$, which is nonnegative  due to \eqref{difference binomial ineq}  and $q\ge 0$. Therefore, applying  (i) to \eqref{lincomparison} implies that $v(x)\geq 0 $ holds for all  $x \in \overline \Omega$, which concludes the proof.
	\end{proof}

	\section{Stability analysis}\label{sec: convergence rates}
	
	The goal of this section is to establish conditional stability estimates for  \eqref{eq: main}. Our analysis is based on the use of Hardy-type inequalities, Gagliardo-Nirenberg inequalities with fractional negative norms, and specific weighted Sobolev spaces with singular weights related to the distance function \eqref{distance function}. All these   analytical tools are collected in the following subsection:

	\subsection{Analytical tools} \label{sec: anal tools} In all what follows, let Assumption \ref{Assump: regularity} hold. Let us first discuss Hardy-type inequalities relevant for our stability analysis. To simplify the notation, we introduce the following distance function:
	\begin{equation}\label{distance function}
		\begin{split}
			\rho(x):=\dist(x,\p\Omega) \quad \forall x\in \overline{\Omega},
		\end{split}
	\end{equation}
	that satisfies the regularity property $\rho\in W^{1,\infty}(\Omega)$ (see, e.g., \cite[Theorem 2.1]{delfour2011shapes}).   For every $p\in (1,\infty)$ and $s\in (-\infty, p-1)$, it is well-known that 
	the so-called weighted Hardy inequality
	\begin{equation}\label{Hardy:weighted}
		\begin{split}
		    \int_{\Omega}  \abs{\varphi(x)}^p \rho(x)^{s-p} \, dx 
		\leq C  \sum_{|\beta|=1}\int_\Omega |D^\beta \varphi(x)|^p \rho(x)^s \, dx 
		\quad \forall \varphi\in C_0^\infty(\Omega)
		\end{split}
	\end{equation}
	holds with a constant $C>0$ depending only on   $\Omega$, $s$, and $p$ (see, e.g., \cite[Theorem 1.2]{koskela2009weighted} and \cite[Theorem 1.6]{nevcas1962methode}). Moreover, for $\sigma\in (0,1)$ and $p\in (1,\infty)$, it was shown in \cite[Theorem 1.1]{dyda2004fractional} that 
	\begin{equation}\label{Hardy:fractional}
		\int_{\Omega}|\varphi(x)|^p \rho(x)^{-\sigma p}\,  dx
		\leq C \int_{\Omega}\int_{\Omega}\frac{|\varphi(x)-\varphi(y)|^p}{|x-y|^{n+\sigma p}} \, dxdy \quad\forall\,\varphi\in C_0^\infty(\Omega).
	\end{equation}
	A useful consequence of   \eqref{Hardy:weighted} and \eqref{Hardy:fractional} is as follows:
	
	\begin{lemma}\label{lemma:hardy}
		Let $p \in (1,\infty)$ and $s \in (0,\infty)$. Then, there exists a constant $C>0$ such that 
		\begin{equation}
			\|\rho^{-s} \varphi\|_{L^p(\Omega)}\leq C \|\varphi\|_{W^{s,p}(\Omega)}   \quad \forall \varphi\in C_0^\infty(\Omega).
		\end{equation}  
	\end{lemma}
	\begin{proof}
		Suppose that $s$ is not an integer. Thus, we have  $s=m+\sigma$ with $m\in \mathbb N$ and $\sigma\in (0,1)$. Let $\varphi\in C_0^\infty(\Omega)$ be arbitrarily fixed.
		Applying the weighted Hardy inequality \eqref{Hardy:weighted} $m$-times, we  infer that 
		\begin{equation} 
			\begin{split}
				\int_{\Omega} |\varphi|^p \rho^{-sp}\,  dx &=\int_{\Omega} \!\! |\varphi|^p\rho^{-mp - \sigma p} \, dx \leq C \sum_{|\beta|=1}  \int_{\Omega} |D^\beta \varphi|^p\rho^{(1-m)p - \sigma p}\,  dx \\
				&\leq { {n} } C^2 \sum_{|\beta|=2}  \int_{\Omega} |D^\beta \varphi|^p\rho^{(2-m)p - \sigma p}\,  dx \le\cdots\leq {{n^{m-1}} }C^m \sum_{|\beta|=m} \int_{\Omega} |D^\beta \varphi|^p \rho^{-\sigma p}\,  dx.
			\end{split}
		\end{equation}
		Then, applying   \eqref{Hardy:fractional} yields  the existence of a   constant $C>0$ such that 
		\begin{equation*}\label{semi norm estimate}
			\begin{split}
				\int_{\Omega} |\varphi|^p \rho^{-sp} \, dx
				\leq C \sum_{|\beta|=m}\int_\Omega\int_\Omega
				\frac{|D^\beta \varphi(x)-D^\beta \varphi(y)|^p}{|x-y|^{n+\sigma p}} \, dxdy \underbrace{\leq}_{\eqref{Wsp:norm}} C \|\varphi\|^p_{W^{s,p}(\Omega)}.
			\end{split}
		\end{equation*}
		The case of $s\in \mathbb N$ is proved in the same way without the above additional estimate. This completes the proof.
	\end{proof}
	Serving as our second analytical tool, we provide Gagliardo--Nirenberg inequalities for functions on the Lipschitz domain $\Omega$  with fractional negative norms. To this end, for $s\in \R$ and $p,p_1\in [1,\infty]$, we denote by $B^{s}_{p,p_1}(\Omega)$  the conventional Besov on $\Omega$ (see \cite[Definition 1.95]{triebel2006theory}).  It is well-known  (see \cite[(5) on page 139]{triebel1983theory}) that   
		\begin{equation}\label{wsp:characterization0}
			W^{s,p}(\Omega)= B^s_{p,p}(\Omega) \quad\forall s\in (0,\infty)\backslash \mathbb N  \quad \forall p \in (1,\infty),
		\end{equation}
		which, in particular, implies   
		$$ 
		W^{s,p}_0(\Omega)={\mathring{B}^s_{p,p}}(\Omega) \quad \forall s\in {(0,\infty)}\backslash \mathbb N \quad  \forall p \in (1,\infty)
		$$ 
		with ${\mathring{B}^s_{p,p}}(\Omega)$ being the closure of 
		$C_0^\infty(\Omega)$ with respect to the norm of $B^s_{p,p}(\Omega)$. Combined with \cite[(43)]{triebel2002function}, this ensures that 
		\begin{align}
			W^{-s,p}(\Omega)= (W_0^{s,p'}(\Omega))^*= {\mathring{B}^s_{p',p'}}(\Omega)^* = B^{-s}_{p,p}(\Omega) \quad \forall s \in {(0,\infty)}\setminus \mathbb N \quad \forall p \in (1,\infty). 
		\end{align}
		Recalling the definition $W^{-k,p}(\Omega)= (W^{k,p'}_0(\Omega))^*$, we conclude  that 
		\begin{equation}\label{wsp:characterization}
			W^{s,p}(\Omega)= B^s_{p,p}(\Omega) \quad\forall s\in \R\backslash { \mathbb Z}  \quad \forall p \in (1,\infty).
		\end{equation}
		\begin{lemma}\label{lemma: interpolation}
			Let $p\in (1,\infty)$ and   $s_0>1$ be real numbers. Then, there exists a constant $C>0$, independent of $u$, such that 
			\begin{equation}\label{lemma:interpolation1}
				\|u\|_{W^{-1,p}(\Omega)}\leq C \|u\|_{W^{-s_0,p}(\Omega)}^{\frac{1}{s_0}} \|u\|_{L^p(\Omega)}^{\frac{s_0-1}{s_0}}
				\quad \forall \varphi\in L^p(\Omega).
			\end{equation}
			Furthermore, for every $s_1>0$, there exists a constant $C>0$, independent of $u$, such that 
			\begin{equation}\label{lemma:interpolation2} 
				\|u\|_{L^p(\Omega)}\leq C \|u\|^{\frac{1}{s_1+1}}_{W^{-s_1,p}(\Omega)} \|u\|^{\frac{s_1}{s_1+1}}_{W^{1,p}(\Omega)}
				\quad \forall u\in W^{1,p}(\Omega).
			\end{equation}
		\end{lemma}
		\begin{proof} 
			By invoking \cite[(1.368)]{triebel2006theory}, we obtain the following interpolation inequalities  
			\begin{align}
				\|u\|_{B^{-1}_{p,1}(\Omega)} &\leq C  \|u\|_{B^{-s_0}_{p,\infty}(\Omega)}^{\frac{1}{s_0}} \|u\|_{B^0_{p,\infty}(\Omega)}^{\frac{s_0-1}{s_0}}\quad \forall u\in   B^{0}_{p,\infty}(\Omega)\label{intep:besov1}\\
				\|u\|_{B^0_{p,1}(\Omega)} & \leq C \|u\|^{\frac{1}{s_1+1}}_{B^{-s_1}_{p,\infty}(\Omega)} \|u\|^{\frac{s_1}{s_1+1}}_{B^{1}_{p,\infty}(\Omega)}
				\quad \forall u\in B^{1}_{p,\infty}(\Omega), \label{intep:besov2}
			\end{align}
			where $C>0$ is a constant independent of $u$. Using \eqref{wsp:characterization}, we deduce from the embedding theorem in \cite[(1.299)]{triebel2006theory} that
			\begin{equation}\label{intep:embed}
				B^{-1}_{p,1}(\Omega) \embed W^{-1,p}(\Omega), 
				\quad W^{-s_0,p}(\Omega)\embed B_{p,\infty}^{-s_0}(\Omega), \quad L^p(\Omega)\embed B^0_{p,\infty}(\Omega).
			\end{equation}
			Combining \eqref{intep:embed} with \eqref{intep:besov1} yields   \eqref{lemma:interpolation1}. Similarly,   \eqref{lemma:interpolation2} follows from \eqref{intep:besov2} together with  
			\begin{equation}
				B^0_{p,1}(\Omega) \embed L^p(\Omega),   
				\quad  W^{-s_1,p}(\Omega)\embed B^{-s_1}_{p,\infty}(\Omega),
				\quad  W^{1,p}(\Omega)\embed B^{1}_{p,\infty}(\Omega). 
			\end{equation}
            This concludes the proof.
		\end{proof}

		Serving as our third analytical tool, we consider the following weighted Sobolev space with a singular weight related to the distance function \eqref{distance function}.
		
		\begin{definition} \label{def: weightedspace}
			Suppose that  $w \in C(\overline{\Omega}) \cap  H^1_0(\Omega) \cap W^{1,\overline p}(\Omega)$ for some  $\overline p \in (2,\infty]$  satisfies  
			\begin{equation}\label{assump:weight}
				\exists C_\gamma,\gamma>0  \quad \forall x \in \overline \Omega: \quad 	w(x)\geq C_\gamma \rho(x)^\gamma  \quad   
			\end{equation}
			with the distance function $\rho(x)$ as defined in \eqref{distance function}.
			Then, for every $\theta\in (1,2]$, we denote by  ${\mathcal{X}}^{\theta }_{w}(\Omega)$  the closure of $C_0^\infty(\Omega)$ under the norm 
			\begin{equation}\label{def:weighted_space}
				\|\varphi\|_{\mathcal{X}^{\theta}_{w}(\Omega)}:=\big\|w^{-m} \varphi\big\|_{W^{1,\theta}(\Omega)} 
				\quad \forall \varphi\in C_0^\infty(\Omega). 
			\end{equation}
		\end{definition}

		\begin{remark} According to Definition \ref{def: weightedspace}, $w$ has a vanishing trace so that the weight function $w^{-m}$ appearing in $\|\cdot\|_{\mathcal{X}^{\theta}_{w}(\Omega)}$ features singularity on the boundary. However, 
			since every $\varphi\in C_0^\infty(\Omega)$ admits a compact support $\textup{supp}(\varphi) \subset \Omega$, the condition \eqref{assump:weight} ensures that 
			\begin{equation}\label{reg w}
				\begin{aligned}
					0 < w^{-m}(x)\leq  C_\gamma^{-m}\rho(x)^{-\gamma m} \leq  C_\gamma^{-m} \textup{dist} (\textup{supp}(\varphi),\p\Omega)^{-\gamma m} \quad \forall x \in \textup{supp}(\varphi) \quad \forall \varphi \in  C_0^\infty(\Omega).
				\end{aligned}
			\end{equation}
			Therefore, $\|\cdot\|_{\mathcal{X}^{\theta}_{w}(\Omega)}$ defines indeed a norm in $ C_0^\infty(\Omega)$. In conclusion,  the closure of $C_0^\infty(\Omega)$ with respect to $\|\cdot\|_{\mathcal{X}^{\theta}_{w}(\Omega)}$ is well-defined.

		\end{remark}

		In view of  $w \in C(\overline{\Omega})$ given by Definition \ref{def: weightedspace},  we see that
		\begin{align*}
			\|\varphi\|_{L^{\theta}(\Omega)}\leq \|w\|_{C(\overline \Omega)}^m \|w^{-m}\varphi\|_{L^{\theta}(\Omega)} \leq \|w\|_{C(\overline \Omega)}^m\|w^{-m}\varphi\|_{W^{1,\theta}(\Omega)}  = 	\|w\|_{C(\overline \Omega)}^m			\|\varphi\|_{\mathcal{X}^{\theta}_{w}(\Omega)} \quad \forall \varphi \in C^\infty_0(\Omega),
		\end{align*} 
		implying by  Definition \ref{def: weightedspace}   that 
		\begin{equation}\label{Gammar embed}
			\mathcal{X}^{\theta}_{w}(\Omega)\embed L^{\theta}(\Omega) \quad\implies 
			\quad L^{\theta'}(\Omega)\embed  \mathcal{X}^{\theta}_{w}(\Omega)^*,
		\end{equation}
        where $\theta'=\frac{\theta}{\theta-1}$ denotes the conjugate exponent.
		To investigate further properties of  $\mathcal{X}^{\theta}_{w}(\Omega)$, we introduce $\mathcal{Z}^\theta_w(\Omega)$ to be the closure  of $C_0^\infty(\Omega)$ under the norm
		\begin{equation}\label{def:Z}
			\|\varphi\|_{\mathcal{Z}^\theta_w(\Omega)}:=\|\rho^{-(m+1)\gamma}\varphi\|_{L^{\vartheta}(\Omega)}+\|\rho^{-m\gamma } \nabla \varphi\|_{L^\theta(\Omega)^n}\quad  \forall \varphi\in C_0^\infty(\Omega) 
		\end{equation} 
		with
		\begin{equation}\label{def:vartheta}
			\vartheta := (1/{\theta} -1/{\overline p})^{-1} \in (1,\infty),
		\end{equation} 
		since $\overline p \in (2,\infty]$ and $\theta \in (1,2]$.  
        
		\begin{lemma}\label{lemma:W dual0}
        Let $w \in C(\overline{\Omega}) \cap  H^1_0(\Omega) \cap W^{1,\overline p}(\Omega)$ for some  $\overline p \in (2,\infty]$ satisfy \eqref{assump:weight} and  $\theta \in (1,2]$. Then, it holds that 
			\begin{equation}\label{W embed Gammar}
				\mathcal{Z}^\theta_w(\Omega)\embed \mathcal{X}^{\theta}_{w}(\Omega).
			\end{equation}
			In other words,   there exists a constant $C>0$ such that 
			$$
			\|\varphi\|_{\mathcal{X}^{\theta}_{w}(\Omega)} \le C\|\varphi\|_{\mathcal{Z}^\theta_w(\Omega)} \quad \forall \varphi \in \mathcal{X}^{\theta}_{w}(\Omega).
			$$
		\end{lemma}
		
		\begin{proof}
			In view of the regularity $w \in C(\overline{\Omega}) \cap W^{1, \overline p}(\Omega)$ and \eqref{assump:weight}, we may apply the chain rule to obtain 
			\begin{equation}
				\begin{aligned}
					\nabla ( w^{-m} \varphi)=-m [\rho^{-(m+1)\gamma}\varphi]\cdot [\rho^{(m+1)\gamma} w^{-{m-1}} \nabla w]  +[\rho^{m\gamma} w^{-m}]\cdot [\rho^{-m\gamma}  \nabla \varphi] \quad  \forall \varphi \in C^\infty_0(\Omega).
				\end{aligned}
			\end{equation}
			Using the above identity as well as
			$$
			\frac{1}{\theta}\underbrace{=}_{\eqref{def:vartheta}} \frac{1}{\vartheta}+\frac{1}{\infty}+\frac{1}{\overline p} \quad \text{and}\quad   \frac{1}{\theta}=\frac{1}{\theta}+\frac{1}{\infty} ,
			$$
			we deduce from the generalized  H\"{o}lder's inequality that 
			\begin{equation}\label{cse:eq0}
				\begin{split}
					\|\nabla(w^{-m}  \varphi)\|_{L^{\theta}(\Omega)} 
					&\leq   m\|\rho^{-(m+1)\gamma}\varphi\|_{L^{\vartheta}(\Omega)}\, \|\rho^{(m+1)\gamma} w^{-m-1}\|_{L^\infty(\textup{supp}(\varphi))} \|\nabla w \|_{L^{\overline{p}}(\Omega)}\\
					&\quad \, +\|\rho^{-m\gamma}\nabla \varphi\|_{L^{\theta}(\Omega)}  \|\rho^{m\gamma} w^{-m}\|_{L^\infty(\textup{supp}(\varphi))}\\
					& \!\!\!\underbrace{\leq}_{\eqref{assump:weight}}  m \|\rho^{-(m+1)\gamma}\varphi\|_{L^{\vartheta}(\Omega)}C_\gamma^{-(m+1)} \|\nabla w \|_{L^{\overline{p}}(\Omega)}+ C_\gamma^{-m}\|\rho^{-m\gamma}\nabla \varphi\|_{L^{\theta}(\Omega)}  \\
					&\!\!\!\underbrace{\leq}_{\eqref{def:Z}} \max\big\{mC_\gamma^{-(m+1)} \|\nabla w \|_{L^{\overline p}(\Omega)},\, C_\gamma^{-m}\big\}\|\varphi\|_{ \mathcal{Z}^\theta_w(\Omega)} \quad \forall \varphi\in C_0^\infty({\Omega})
				\end{split}
			\end{equation}
			and 
			\begin{equation}\begin{split}
					\|w^{-m}  \varphi\|_{L^{\theta}(\Omega)}&\leq  \|\rho^{-(m+1)\gamma}\varphi\|_{L^{\theta}(\Omega)}  \|\rho^{(m+1)\gamma} w^{-(m+1)}\|_{L^\infty(\textup{supp}(\varphi))}\|w\|_{C(\overline \Omega)} \\ 
					&\leq C_\gamma^{-(m+1)}\|w\|_{C(\overline \Omega)} \norm{\varphi}_{\mathcal{Z}^\theta_w(\Omega)} \quad \forall \varphi\in C_0^\infty({\Omega}).
				\end{split}
			\end{equation}
			In conclusion, the continuous embedding \eqref{W embed Gammar} holds. 
		\end{proof}

		\subsection{Conditional stability estimates}
		
		For the rest of this article, let Assumption \ref{Assump: regularity} hold and $\overline p\in (2,\infty]$ be as in Lemma \ref{thm: well-posed N}.  For every given nonnegative coefficient  $q \in L^r(\Omega)$ of the nonlinearity, we make use of the standing notation $u=u(q) \in H^1_0(\Omega)\cap C(\overline \Omega) \cap W^{1,\overline p}(\Omega)$ to denote the unique weak solution of \eqref{eq: main}. Furthermore, in the sequel, we consider the following two admissible sets for the coefficient $q$:  
		\begin{equation}  
			    \mathcal{A}:= \left\{q\in L^\infty(\Omega)  : \, { \underline q \leq q\leq \overline q } \ \textrm{ a.e. in  }   \Omega\right\} \label{def:adimisble}
			\  \textrm{and} \ 
			\widetilde{\mathcal{A}}_p:=  \{q \in \mathcal{A} \cap W^{1,p}(\Omega) :    \| q\|_{W^{1,p}(\Omega)}\leq \widetilde M  \}
		\end{equation}
		for some prescribed constants  {$0 \le \underline q \le \overline q  < \infty$}, $p\in [2,\infty)$ and $\widetilde M>0$.

		\begin{assumption} \label{Assump: cse}  Let Assumption \ref{Assump: regularity} be satisfied. Suppose that the   weak solution $ u(q^\dag)$  of \eqref{eq: main} associated with the true coefficient $q^\dag \in \mathcal A$ satisfies 
			\begin{equation}\label{positive}
				\exists C_\gamma,\gamma>0 \quad \forall x \in \overline \Omega: \quad	u(q^\dag)(x)\geq C_\gamma \rho(x)^\gamma  \quad    \forall x \in \overline \Omega
			\end{equation}
			with the distance function $\rho(x)$ as defined in \eqref{distance function}.
			
		\end{assumption}
		
		
		\begin{remark}\label{remark:trivail}
			Note that Assumption \ref{Assump: cse} is inspired by \cite{jin2022convergence}, who focused on the linear case  
			\begin{equation}\label{eq: main3}
				    -\Delta u + qu =f\quad \text{ in }\Omega\quad\textup{and}\quad
				u  =0 \quad\text{ on }\p \Omega,
			\end{equation}
			and  proved in \cite[Theorem 2.1]{jin2022convergence} the following  conditional stability estimate:  
			\begin{equation}\label{cselinear}
				\|q-q^\dag\|_{L^2(\Omega)}\leq C \|u(q)-u(q^\dag)\|_{H^1_0(\Omega)}^{\frac{1}{2+4\gamma}}
				\quad\forall q\in \widetilde{\mathcal{A}}_2.
			\end{equation}
			In the upcoming theorem, we prove our main result on the conditional stability estimate for the nonlinear case \eqref{eq: main}. In particular, for $m=1$, our result improves the convergence order \cite[Theorem 2.1]{jin2022convergence} to the order $\frac{1}{1+2\gamma}$. 
		\end{remark}

		Note that Assumption \ref{Assump: cse}  allows us to set $w=u^\dag:=u(q^\dag)$ in Definition \ref{def: weightedspace}. Furthermore, for  any given $p \in   [2,\overline p]\cap[2,\infty)$, its conjugate exponent satisfies $p' = \frac{p}{p-1}\in (1,2]$ such that the choice $\theta=p'$ is also allowed for Definition \ref{def: weightedspace}. Altogether, employing the weighted space $\mathcal X^{p'}_{u^\dag}(\Omega)$ and all analytical tools from the previous section, we are in a position to prove the following key lemma:
		
		\begin{lemma}\label{lemma:W dual}
			Let Assumptions \ref{Assump: regularity} and \ref{Assump: cse} be satisfied    and $\kappa:= m\gamma+\gamma+n/\overline{p}$ with $\overline{p}$ from Lemma \ref{thm: well-posed N}. Then, for every $p \in [2,\overline{p}] \cap [2,\infty)$,  there exists a constant $C>0$  such that 
			\begin{align}
				\|\varphi\|_{W^{-1,p}(\Omega)}&\leq\begin{cases}
					C\|\varphi\|_{\mathcal{X}^{p'}_{u^\dag}(\Omega)^*} \quad\quad\quad\quad\,\,\,\,\forall \varphi\in    L^{p}(\Omega)  \quad&\textup{if}\quad \kappa\in (0,1], \\ 
					C\|\varphi\|_{\mathcal{X}^{p'}_{u^\dag}(\Omega)^*}^{\frac{1}{\kappa}}\|\varphi\|_{L^{p}(\Omega)}^{\frac{\kappa-1}{\kappa}}
					\quad \forall \varphi\in L^{p}(\Omega) 
					\quad&\textup{if}\quad \kappa>1,
				\end{cases}\label{Gammar:interp1}
			\end{align}
			and 
			\begin{align}
				\|\varphi\|_{L^{p'}(\Omega)}\leq
				C\|\varphi\|_{\mathcal{X}^{p'}_{u^\dag}(\Omega)^*}^{\frac{1}{1+\kappa}}\|\varphi\|_{W^{1,p}(\Omega)}^\frac{\kappa}{1+\kappa} \quad \forall \varphi\in W^{1,p}(\Omega).\label{Gamma:interp2} 
			\end{align}
	
		\end{lemma}

		\begin{proof} 
			Let $p \in [2,\overline{p}]\cap[2,\infty)$ and $\vartheta :=(1/{p'} -1/{\overline p})^{-1} \in (1,\infty)$. Then, Lemma \ref{lemma:hardy} yields 
			\begin{align}
				\|\rho^{-(m+1)\gamma}\varphi\|_{L^{ \vartheta}(\Omega)}&\leq C\|\varphi\|_{W^{(m+1)\gamma, \vartheta}(\Omega)} \quad \forall \varphi\in C_0^\infty(\Omega) \\
				\|\rho^{-m\gamma}\nabla \varphi\|_{L^{p'}(\Omega)^n}&\leq C\|\varphi\|_{W^{m\gamma+1,p'}(\Omega)} \quad \forall \varphi\in C_0^\infty(\Omega)
			\end{align}
			with a constant $C>0$ independent of $\varphi$. Thus, along with the definition \eqref{def:Z} with $\theta=p'$, it follows that 
			$$
			\|\varphi\|_{\mathcal{X}^{p'}_{u^\dag}(\Omega)} \le C\big( \|\varphi\|_{W^{(m+1)\gamma, \vartheta}(\Omega)} +\|\varphi\|_{W^{m\gamma+1,p'}(\Omega)}\big) \quad \forall \varphi\in C_0^\infty(\Omega),
			$$
			implying
			\begin{equation}\label{embeddingZ}
				W_0^{(m+1)\gamma,\vartheta}(\Omega) \cap W_0^{m\gamma+1, p'}(\Omega) \embed  \mathcal{Z}^{p'}_{u^\dag}(\Omega) \embed \mathcal{X}^{p'}_{u^\dag}(\Omega).
			\end{equation}
			By definition, we have $\kappa= (m+1)\gamma +n/\overline{p}$ and $1/\vartheta =  1/{p'} -1/{\overline p}$, and so 
			$
			\kappa - n/p'=   (m+1)/\gamma - n/\vartheta.
			$
			Thus, the Sobolev embedding theorem (see \cite[(1.301)]{triebel2006theory}) implies $ W_0^{\kappa, p'}(\Omega)\embed W^{(m+1)\gamma,\vartheta}_0(\Omega)$. In combination with \eqref{embeddingZ}, we obtain that
			$$
		W^{\kappa,p'}_0(\Omega)\embed W_0^{(m+1)\gamma,\vartheta}(\Omega)\cap W_0^{m\gamma+1, p'}(\Omega) \embed \mathcal{X}^{p'}_{u^\dag}(\Omega),
			$$
			and so
			\begin{equation}\label{W:embed2}
				\mathcal{X}^{p'}_{u^\dag}(\Omega)^* \embed 	W^{-\kappa,p}(\Omega).
			\end{equation}
			In the case of $\kappa\in (0,1]$, the estimate  \eqref{Gammar:interp1} follows directly from  the  embeddings \eqref{W:embed2} and \eqref{Gammar embed} (with $\theta=p'$ and $w= u^\dag$) as well as  $W^{-\kappa,p}(\Omega)\embed W^{-1,p}(\Omega)$.  If $\kappa >1$, it holds that 
			
			\begin{equation}
				\|\varphi\|_{W^{-1,p}(\Omega)}
				\underbrace{\leq}_{\text{Lem. \ref{lemma: interpolation}} } C \|\varphi\|_{W^{-\kappa,p}(\Omega)}^{\frac{1}{\kappa}} \|\varphi\|_{L^{p}(\Omega)}^{\frac{\kappa-1}{\kappa}}\underbrace{\leq}_{\eqref{W:embed2}} C\|\varphi\|_{\mathcal{X}^{p'}_{u^\dag}(\Omega)^*}^{\frac{1}{\kappa}}\|\varphi\|_{L^{p}(\Omega)}^{\frac{\kappa-1}{\kappa}}
				\quad \forall \varphi\in L^{p}(\Omega), 
			\end{equation}
			which is exactly \eqref{Gammar:interp1} for the case of $\kappa>1$. By 
			\eqref{W:embed2} and inequality \eqref{lemma:interpolation2} (with $s_1=\kappa$),  
			the estimate \eqref{Gamma:interp2} follows.
		\end{proof}

		\begin{theorem}\label{thm: cse}
			Let Assumptions \ref{Assump: regularity} and \ref{Assump: cse} be satisfied and $\kappa:= m\gamma+\gamma+n/\overline p$ with $\overline{p}$ as in Lemma \ref{thm: well-posed N}.  
            Then,  for every $p \in [2,\overline p]\cap[2,\infty)$, there exists a constant $C>0$  such that 
			\begin{align}
				&\big\|q-q^\dag \big\|_{W^{-1,p}(\Omega)}\leq
				\begin{cases}
					C \big\|u(q)-u(q^\dag)\big\|_{W^{1,p}(\Omega)} \quad \forall q\in \mathcal{A} \quad\textup{if } \kappa\in (0,1],\\
					C \big\|u(q)-u(q^\dag)\big\|^{\frac{1}{\kappa}}_{W^{1,p}(\Omega)}  \quad \forall q\in \mathcal{A} \quad\textup{if } \kappa\in (1,\infty), \label{cse:W-1p}
				\end{cases}\\
				&\big\|q-q^\dag \big\|_{L^p(\Omega)}  \leq C\big\|u(q)-u(q^\dag)\big\|^{\frac{1}{1+\kappa}}_{W^{1,p}(\Omega)}  \quad \forall q\in \widetilde{\mathcal{A}}_p. \label{eq:cse:Lp}
			\end{align} 
			
		\end{theorem}

		\begin{proof} 
			Let $p \in [2,\overline p]\cap[2,\infty)$ and $q \in \mathcal A$. 
			From the weak formulations of $u(q)$ and $u(q^\dag)$, it follows for all $v\in W^{1,p'}_0(\Omega)$ that 
			\begin{equation}\label{weak id2 in sec 4}
				\begin{split}
					& \int_{\Omega} (q^\dag-q) u(q^\dag)^m v \, dx =\int_{\Omega} \big[  \sigma \nabla (u(q)-u(q^\dag)) \cdot \nabla v   +  q (u(q)^m-u(q^\dag)^m) v \big] \, dx \\
					&\quad \underbrace{=}_{\eqref{difference binomial}} \int_{\Omega} \sigma \nabla (u(q)-u(q^\dag)) \cdot \nabla v + {q}\bigg( \sum_{k=0}^{m-1}u(q)^k u(q^\dag)^{{m-1-k}} \bigg)(u(q)-u(q^\dag)) v \,  dx.
				\end{split}
			\end{equation}
			Furthermore, by Lemma \ref{thm: well-posed N},   \eqref{def:adimisble},  we have 
			$C_{\mathcal{A}}:=\sup\limits_{q\in \mathcal{A}} \|u(q)\|_{C(\overline{\Omega})}  < \infty$, which together with \eqref{weak id2 in sec 4} and \eqref{ellipticity} implies that 
			\begin{equation}\label{weak id3 in sec 4}
				\begin{split}
					\int_{\Omega} (q^\dag-q) u(q^\dag)^m v \, dx &
					\leq   \underbrace{\big(n\max_{1\leq i,j\leq n}\|\sigma_{i,j}\|_{L^\infty(\Omega)}  +m\overline{q} C^{m-1}_\mathcal{A}\big)}_{=:\hat C}\|u(q)-u(q^\dag)\|_{W^{1,p}(\Omega)}\| v\|_{W^{1,p'}(\Omega)} 
				\end{split}
			\end{equation}
			holds for all $ v\in W_0^{1,p'}(\Omega)$. In view of Assumption \ref{Assump: cse} and since $p' = \frac{p}{p-1} \in (1,2]$, the weighted space $\mathcal{X}^{p'}_{u^\dag}(\Omega)$ as in Definition \ref{def: weightedspace} is well-defined for $w=u(q^\dag)$ and $\theta= p'$.  For every arbitrarily fixed $\varphi \in \mathcal{X}^{p'}_{u^\dag}(\Omega)$,  the construction of $\mathcal{X}^{p'}_{u^\dag}(\Omega)$ (see \eqref{def:weighted_space} with $w=u(q^\dag)$ and $\theta=p'$) implies the existence of a sequence of $\{\varphi_k\}_{k\in \mathbb N}\subset C_0^\infty(\Omega)$ such that $u(q^\dag)^{-m}\varphi_k\in W_0^{1,p'}(\Omega)$ and  $\varphi_k\to \varphi$ in  $\mathcal{X}^{p'}_{u^\dag}(\Omega)$ as $k\to \infty$. Thus, we may set  $v=u(q^\dag)^{-m} \varphi_k$ in \eqref{weak id3 in sec 4} to obtain that
			\begin{equation}\label{pre:condi12:0}
				\left|\int_{\Omega} (q-q^\dag)\varphi_k \, dx \right|\leq   \hat C \big\|u(q)-u(q^\dag)\big\|_{W^{1,p}(\Omega)}  \|\varphi_k\|_{\mathcal{X}^{p'}_{u^\dag}(\Omega)} \quad \forall  k \in \mathbb N \quad  \forall q \in \mathcal A.
			\end{equation} 
			Thus, in view of the continuous embedding \eqref{Gammar embed}, we may pass to the limit $k \to \infty$ in \eqref{pre:condi12:0} to obtain
			\begin{equation}\label{pre:condi12}
				\left|\int_{\Omega} (q-q^\dag)\varphi \, dx \right|\leq \hat C  \big\|u(q)-u(q^\dag)\big\|_{W^{1,p}(\Omega)}  \|\varphi \|_{\mathcal{X}^{p'}_{u^\dag}(\Omega)} \quad \forall  q\in \mathcal A.
			\end{equation}
			Since $\varphi \in \mathcal{X}^{p'}_{u^\dag}(\Omega)$ was chosen arbitrarily, it follows that  
			\begin{equation}\label{cse:Gamma}
				\|q-q^\dag\|_{\mathcal{X}^{p'}_{u^\dag}(\Omega)^*} = \sup_{ \varphi\neq0} \frac{\left|\int_{\Omega} (q-q^\dag)\varphi \, dx \right|}{ \|\varphi \|_{\mathcal{X}^{p'}_{u^\dag}(\Omega)}}\leq \hat C   \big\|u(q)-u(q^\dag)\|_{W^{1,p}(\Omega)} \quad\forall q\in \mathcal{A}.
			\end{equation} 
			
			On the one hand, for the case $\kappa\in (0,1)$, the estimate \eqref{cse:W-1p} follows immediately from 
			\eqref{cse:Gamma} and  \eqref{Gammar:interp1}.       
			On the other hand, for the case $\kappa>1$,  we have that  \begin{equation}
				\begin{split}
					\|q-q^\dag\|_{W^{-1,p}(\Omega)}&\underbrace{\leq}_{\eqref{Gammar:interp1}} C \|q-q^\dag\|_{\mathcal{X}^{p'}_{u^\dag}(\Omega)^*}^{\frac{1}{\kappa}}\|q-q^\dag\|_{L^{p}(\Omega)}^{\frac{\kappa-1}{\kappa}} 
					\underbrace{\leq}_{\eqref{cse:Gamma}}  C  \hat C^{\frac{1}{\kappa}} \|u(q)-u(q^\dag)\big\|^{\frac{1}{\kappa}}_{W^{1,p}(\Omega)}  \|q-q^\dag\|_{L^{p}(\Omega)}^{\frac{\kappa-1}{\kappa}} \\
					&\ \, \leq C  \hat C^{\frac{1}{\kappa}} 2\overline{q}|\Omega|^{\frac{\kappa-1}{\kappa}}\|u(q)-u(q^\dag)\big\|^{\frac{1}{\kappa}}_{W^{1,p}(\Omega)} \quad \forall q \in \mathcal A, 
				\end{split}
			\end{equation}
			leading to \eqref{cse:W-1p}.  Similarly, the combination of \eqref{cse:Gamma} and
			\eqref{Gamma:interp2} shows for all $q\in \widetilde{\mathcal{A}}$ that 
			\begin{align}
				\|q-q^\dag\|_{L^{p}(\Omega)}&\underbrace{\leq}_{\eqref{Gamma:interp2}}
				C\|q-q^\dag\|_{\mathcal{X}^{p'}_{u^\dag}(\Omega)^*}^{\frac{1}{1+\kappa}}\|q-q^\dag\|_{W^{1,p}(\Omega)}^\frac{\kappa}{1+\kappa}  
				\underbrace{\leq}_{ \eqref{cse:Gamma}} C \hat C^{\frac{1}{1+\kappa}} (2\widetilde M)^\frac{\kappa}{1+\kappa} \|u(q)-u(q^\dag)\big\|_{W^{1,p}(\Omega)}^{\frac{1}{1+\kappa}},
			\end{align}
			where we have used the uniform boundedness of $q\in \widetilde{\mathcal{A}}_p$ in $W^{1,p}(\Omega)$. This gives \eqref{eq:cse:Lp} and hence completes the proof.
		\end{proof}

		\subsection{Verification of Assumption \ref{Assump: cse}}	
		
		This subsection is devoted to the verification of Assumption \ref{Assump: cse} on a convex domain. The 
		proof relies on the following pointwise lower bound estimate. 
		
		\begin{lemma}\label{lemma:low:est}  Let Assumption \ref{Assump: regularity}  hold and   $q \in C(\overline \Omega)\cap C^\vartheta(\Omega)$ for some $\vartheta >0$   be nonnegative. Suppose additionally that $\Omega$ is   convex,   $f \in C(\overline \Omega)\cap C^\vartheta(\Omega)$ is non-negative, $m\ge2$, and   $\sigma  \in C^{\infty}({\Omega})^{n\times n}$. Then, the unique weak solution $u(q)$ to \eqref{eq: main} satisfies
			\begin{equation}
				u(q)(x)\geq H(x)\bigg(1+(m-1)\int_{\Omega} G(x,y) H(y)^m q(y)\,  dy\bigg)^{-\frac{1}{m-1}}\quad \forall x\in \Omega,
			\end{equation}
			where $H(x):=\int_{\Omega} G(x,y)f(y)\, dy$.
		\end{lemma}
		\begin{proof}
			By  Lemma \ref{thm: well-posed N}, it follows that $u(q) \in  H^1_0(\Omega) \cap W^{1,\infty}(\Omega)$.
			Thanks to the regularity properties  $u(q) \in W^{1,\infty}(\Omega)$ and $q,f \in C(\overline \Omega)\cap C^\vartheta(\Omega)$,  we see that  the weak solution $u(q)$ satisfies 
			\begin{equation}\label{sysuinf2}
				-\nabla \cdot (\sigma \nabla {u}(q))= f - {q} \,{u}(q)^m  \in  C(\overline \Omega)\cap C^\vartheta(\Omega)  
			\end{equation}
			Consequently, Schauder's estimate \cite[Thm 6.13]{GT2001_elliptic} implies  $u(q)\in C^{2,\theta}(\Omega)$.  Altogether, we may apply \cite[Theorem 3.1]{grigor2019pointwise}   to the   solution   $u(q)\in C^{2,\theta}(\Omega)$  and complete the proof. 	    
		\end{proof}
		\begin{remark}
			The smooth requirement $\sigma\in C^{\infty}(\overline{\Omega})^{n \times n}$ in the above lemma can be weakened to $\sigma\in W^{1,\beta}(\Omega)^{n\times n}$ for some $\beta>n$. This is possible by 
			a perturbation argument and the pointwise estimates of Green functions on convex domains by 
			\cite{gruter1982green}. 
		\end{remark}
		
		\begin{proposition}\label{prop: D1} Let Assumption \ref{Assump: regularity}  hold.
			Assume additionally that  $\Omega$ is convex with dimension $n=2,3$,   $\sigma\in C^{\infty}(\overline{\Omega})^{n \times n}$,  and   there exists a constant $c_f>0$  such that $f(x)\geq c_f$ holds for a.e. $x \in \Omega$. Then, there exists $c^\dag>0$ such that   
			\begin{equation}\label{low:uq}
				u(q)(x)\geq \begin{cases}
					c^\dag  \rho(x)^{2+\frac{1}{m-1}} & \text{ for all } x\in \overline \Omega, \  q \in \mathcal A, \text{ as }  m>1 \\
					c^\dag  \rho(x)^2 &\text{ for all } x\in \overline \Omega, \ q \in \mathcal A ,\text{ as } m=1.
				\end{cases}.
			\end{equation}
			In particular, Assumption \ref{Assump: cse} is satisfied.
		\end{proposition}   
		{\color{blue}

		}
		
		\begin{proof}[Proof]
			Since the case $m=1$ has been treated in \cite{jin2022convergence} (i.e., the linear case), we will prove the rest.
			Let us divide the proof into three steps:
			
			\noindent (Step 1.)	Let $q \in \mathcal A$ be arbitrarily fixed. By
			$\overline{u}, u_c  \in H^1_0(\Omega) \cap C(\overline \Omega)$,  we  denote, respectively,  the unique  weak solutions to 
			\begin{equation}\label{sysuinf}
				\begin{cases}
					-\nabla \cdot (\sigma \nabla \overline{u})+ \overline q \,\overline{u}^m = c_f &\text{in }\Omega \\
					\overline{u}=0 &\text{on }\p \Omega
				\end{cases},\quad
				\begin{cases}
					-\nabla \cdot (\sigma \nabla {u}_c)+ q u_c^m = c_f &\text{in }\Omega \\
					{u}_c=0 &\text{on }\p \Omega
				\end{cases} .
			\end{equation}
			Since $c_f > 0$, the maximum principle (Lemma \ref{lemma:MP}) applied to above systems yield   
			\begin{equation}\label{nonnegative}
				\overline{u} (x) >0 \quad \text{and} \quad    {u}_c(x)  >0 \quad \forall x \in \Omega.
			\end{equation}
			Moreover, since $f(x) \ge c_f $ holds for a.e. $x \in \Omega$, the comparison principle (Lemma \ref{lemma:MP} \ref{item 3 MP}) implies  that 
			\begin{equation}\label{uqunder}
				u(q)(x)\geq {u}_c(x)   \quad \forall x \in \overline   \Omega. 
			\end{equation}
			Now, we see that   $e:={u}_c-\overline{u}  \in H^1_0(\Omega) \cap C(\overline \Omega)$ is the unique weak solution to 
			\begin{equation}\label{systemeu}
					-\nabla \cdot (\sigma \nabla e)+  q \bigg(\sum_{k=0}^{m-1} {u}_c^k \overline{u}^{m-1-k}\bigg) e = (\overline q-q) \overline{u}^m  \quad \text{ in }\Omega \quad\textup{and}\quad
					e =0 \quad \text{ on }\p \Omega.
			\end{equation}
			In view of  \eqref{nonnegative} and since  $q(x)\le \overline q$ holds for a.e. $x \in \Omega$,  we have
			\begin{equation}\label{systeminf}
				(\overline q(x)- {q}(x)) \overline {u}^m(x)  \ge 0  \textrm{ for a.e. } x\in \Omega.
			\end{equation}
			Thus, applying
			the maximum principle   (Lemma \ref{lemma:MP})  to \eqref{systemeu} gives
			\begin{equation}
				e_u(x)\geq 0 \quad  \forall x \in  \overline  \Omega   \implies   {u}_c(x)\ge\overline{u}(x), \quad    \forall x \in  \overline  \Omega.  
			\end{equation}
			Combining the above inequality with \eqref{uqunder} and since $q \in \mathcal A$ was chosen arbitrarily, we infer   that 
			$$
			{u}(q)(x)\ge\overline{u}(x), \quad    \forall x \in  \overline  \Omega \quad \forall q \in \mathcal A. 
			$$
			In conclusion, the claim \eqref{low:uq}  is a consequence of
			\begin{equation}\label{low:uqinf}
				\overline{u}(x) \geq   c^\dag \rho(x)^{2+\frac{1}{m-1}}  \quad \forall x\in \overline \Omega,
			\end{equation}
			for some constant $c^\dag >0$. We will then prove \eqref{low:uqinf} by the following steps.

			\noindent (Step 2.) Let us denote by $G$ the   Green function (see \cite{fromm1993potential}) for the elliptic operator $-\nabla \cdot \sigma \nabla: W^{1,p}_0(\Omega) \to  W^{1,p'}_0(\Omega)^*$ with $p\in (2,\overline p)$ as in  Lemma \ref{lemma: Meyer-Jerison-Kenig}. Furthermore, we set
			\begin{equation}\label{H}
				H\in C(\overline \Omega) \cap W^{1,p}_0(\Omega), \quad  H(x):=\int_{\Omega} G(x,y) c_f \, dy \quad \forall x\in  \Omega,
			\end{equation}
			where the regularity $H\in C(\overline \Omega)$ is obtained due to the construction that it is the unique weak solution of 
			$$
			-\nabla \cdot (\sigma \nabla H) = c_f\quad\text{ in } \Omega \quad\textup{and} \quad
			H=0 \quad\text{ on } \partial \Omega.
			$$
			Moreover, one can ensure $H\in C^2(\Omega)$ using the Schauder estimate introduced in \cite[Chapter 6]{GT2001_elliptic}.
			
			Let us show that 
			\begin{equation}\label{claim1}
				\exists C_f \in (0,\infty) \quad \forall x\in  \Omega  : \quad  	C^{-1}_f\rho(x)^2  \leq  H(x)\leq C_f.  
			\end{equation}
			To this end, on the one hand, the right-hand side estimate in \eqref{claim1} obviously holds for any constant being larger than or equal to $\|H\|_{C(\overline \Omega)}$.  On the other hand, since $n\geq 2$,     \cite[Theorem 2.2]{jin2022convergence} implies the existence of a constant $C_{\Omega}>0$ such that 
			\begin{align}
				G (x,y)\geq  \frac{C_\Omega^{-1}}{|x-y|^{n-2}}   \quad \forall  x,y\in  \Omega  \quad  \text{with}\quad  
				|x-y|\leq \frac{1}{2}\rho(x). 
			\end{align}
			The above estimate, along with the nonnegativity of the Green function $G$, implies finally 
			\begin{equation}
				H(x) \geq \frac{c_f}{C_\Omega} \int_{B(x,\rho_x)} |x-y|^{2-n}\,  dy= \frac{c_f}{8C_\Omega}  \rho(x)^2 |\mathbb S_{n-1}|  \quad \forall x \in \Omega,
			\end{equation}
			where $B(x,\frac{1}{2}\rho_x)$  denotes the open ball with radius $\frac{1}{2}\rho_x$ centered at $x$, and $|\mathbb S_{n-1}|$ is the Lebesgue measure of $(n-1)$-dimensional unit sphere. This proves the claim \eqref{claim1} with $
			C_f := \max\big(\|H\|_{C(\overline \Omega)},\,  \frac{c_f}{8C_\Omega} |\mathbb S_{n-1}| \big)$.\\

			\noindent (Step 3) The application of Lemma \ref{lemma:low:est} shows that 
			\begin{equation}\label{uq1}
				\begin{split}
					\overline{u}(x)\geq H(x) \bigg(1+(m-1)\frac{1}{H(x)}\int_{\Omega} G(x,y) H^m(y)\overline q\,  dy \bigg)^{-\frac{1}{m-1}},  \quad\forall x\in \Omega.
				\end{split}
			\end{equation}
			By the definition of the Green function, the function $w:=\int_{\Omega} G (\cdot,y) H^m(y)\overline q\, dy$
			solves the linear elliptic equation 
			$$
				-\nabla \cdot ( \sigma \nabla w) = H^m \overline q \quad \text{ in } \quad\Omega \quad \textup{and}\quad 
				w =0 \quad \text{ on }\quad \partial\Omega. 
			$$ 
			Thus, since $ H^m \overline q \in C(\overline\Omega)$, $\Omega$ is convex, and $\sigma \in C^{\infty}(\Omega)^{n\times n}$, it follows that $w\in W^{1,\infty}(\Omega)$, and then the fact that $w$ vanishes on $\p\Omega$ further implies  
			\begin{equation}
				\begin{split}
					w(x)\leq \norm{\nabla w}_{L^\infty(\Omega)} \rho(x)  \leq C_{p,r}\left\|H^m \overline q \right\|_{L^r(\Omega)}\rho(x) , \quad \text{for all }x\in \Omega,
				\end{split}
			\end{equation}
			where $C_{p,r}>0$ is the constant given in the estimate \eqref{est: well 1}.  
			Combing this with \eqref{claim1}, we have that 
			\begin{equation}
				\frac{1}{H(x)}\int_{\Omega} G (x,y) H^m(y)\overline q  \, dy\leq \frac{C_fC_{p,r}\left\|H^m \overline q \right\|_{L^r(\Omega)}}{\rho(x)} \quad \forall x\in \Omega, 
			\end{equation}
			which, together with \eqref{uq1}, shows that 
			\begin{equation}
				\begin{split}
					\overline{u}(x)\geq c_0\rho(x)^{2+\frac{1}{m-1}}\left(c_0\rho(x)+(m-1)c_1)\right)^{-\frac{1}{m-1}}\geq c^\dag \rho(x)^{2+\frac{1}{m-1}},
				\end{split}
			\end{equation}
			where $c_1= C_{p,r}\left\|H^m \overline q \right\|_{L^r(\Omega)}$. 
			Here, the constant $c^\dag>0$ depends on $c_0$, but is independent of the value $m$.
			This completes the proof, since 
			$\rho(x)\leq \mathrm{diam}(\Omega)< \infty$ is uniformly bounded.
		\end{proof}

		\section{Finite element approximations}\label{sec: finite element}
		This section analyzes the numerical reconstruction of $q$ based on the piecewise linear FEM as readily proposed in the introduction. Let us begin by introducing the standing assumption for our numerical analysis:

		\begin{assumption} \label{Assump: FEM} Let Assumption \ref{Assump: regularity} be satisfied with $\Omega$ being either a bounded convex polygonal domain in $\R^2$, or a  bounded convex polyhedral domain in  $\R^3$.  let additionally   $\sigma\in W^{1,\infty}(\Omega)^{n\times n}$ and $r>n$.  Furthermore, the true coefficient $q^\dag$  is assumed to satisfy  
        \begin{equation}\label{eq:assump:qdag}
            		q^\dag  \in H^{1}(\Omega) \quad \textrm{and} \quad  {\underline  q \leq q^\dag(x) \leq \overline q} \quad \textrm{ for  a.e. } x \in \Omega. 
        \end{equation}
 
		\end{assumption}
        \begin{remark} \label{H1 ass only}
             The proposed regularity condition for $q^\dag$ for our error analysis is significantly weaker than the $H^2(\Omega)$-regularity assumption considered in  \cite{jin2022convergence} for the linear case $m=1$. This relaxation is achieved by employing the Carstensen quasi-interpolation operator \cite{carstensen1999quasi} in combination with the $L^2$ projection operator and corresponding error estimates in $H^{-1}(\Omega)$. 
        \end{remark}
		From Lemma \ref{thm: well-posed N}, it follows that     Assumption \ref{Assump: FEM} implies that $u(q) \in H^1_0(\Omega) \cap H^2(\Omega) \cap W^{1,\infty}(\Omega)$ holds for all $q \in \mathcal A$.
		Supposing that Assumption \ref{Assump: FEM} holds,  in all what follows, let $\{\mathcal{T}_h\}_{h>0}$  be  a quasi-uniform family of simplicial triangulations of $\Omega$    and
		\begin{equation}\label{FEM}
			S_h:=\left\{u_h\in C(\overline{\Omega})   : \,  \left. u_h\right|_T \,\textup{is a polynomial of degree 1 on}\, T \textrm{ for all } T\in \mathcal{T}_h\right\},\quad
			S_h^0:= S_h \cap H^1_0(\Omega).
		\end{equation}
		In view of the quasi uniformity of $\left\{\mathcal{T}_h\right\}_{h>0}$,   the following inequalities  (see \cite[Corollary 1.141]{ern2004theory}) are obtained:
		\begin{equation}\label{inverse:property}
			\|v_h\|_{H^{l}(\Omega)}\leq C h^{-(l-k)}\|v_h\|_{H^{k}(\Omega)}\quad\forall 
			v_h\in S_h \quad {\forall 0\leq k\leq l \le 1.} 
		\end{equation}
		We consider the finite element approximation of \eqref{eq: main}  as follows:
		Given $q\in \mathcal{A}$, let $u_h(q) \in S_h^0$ denote  the unique solution to
		\begin{equation}\label{numerical:op}
			a( u_h(q),  v)+ \big(q u_h(q)^m, v\big)_{L^2(\Omega)}=(f,v)_{L^2(\Omega)} \quad \forall v\in S_{h}^0 
		\end{equation}
		with  $a: H^1_0(\Omega)\times H^1_0(\Omega) \to \R$, $a(u,v):=\int_\Omega \sigma \nabla u\cdot \nabla v \, dx$.
		The well-posedness of the discrete variational problem \eqref{numerical:op} follows standard arguments (cf. \cite[p. 3]{as2002uniform}).  

On the basis of \eqref{numerical:op}, we  numerically reconstruct  $q^\dag$, defined by  a minimizer $\qda \in \mathcal A_h:= \mathcal A \cap S_h$ (see \eqref{def:adimisble}) of the following least squares minimization problem:
\begin{equation}\label{tik dis}
			 \min_{q_h\in \mathcal{A}_h}\mathcal{J}_{\alpha,h}^\delta(q_h) :=\frac{1}{2}\big\|u_h(q_h)-y^\delta\big\|^2_{L^2(\Omega)}+\frac{\alpha}{2}(\|q_h\|_{L^2(\Omega)}^2+\|\nabla q_h\|^2_{L^2(\Omega)})
		\end{equation}
		with noise measurement  $y^\delta\in L^2(\Omega)$ satisfying 
		\begin{equation}\label{noisy} 
			\big\|u(q^\dag)-y^\delta\big\|_{L^2(\Omega)}\leq \delta. 
		\end{equation}
      Note that the existence of a minimizer $q^\delta_{q,h}$   for every fixed  $h,\delta, \alpha$ follows from the Weierstrass lemma, as the feasible set $\mathcal A_h$ is compact and $u_h:\mathcal{A}_h\subset S_h\to S_h$ is continuous (see Remark \ref{remark:conti:uh}).

        \subsection{Error estimates} Our goal is to establish error estimates between the finite element approximation $\qda$   and the true solution $q^\dag$. Let us start by introducing the operators $A: H^1_0(\Omega)\to H^{-1}(\Omega)$ and  $A_h: S_h^0\to S_h^0$ associated with the bilinear form $a$, respectively, for the continuous and discrete setting by
		\begin{align}
			\langle Au,v \rangle_{H^{-1}(\Omega),H^1_0(\Omega)}:=a(u,v) \quad  \forall\, u,v\in H_0^1(\Omega),\quad 
			(A_h w_h,v_h)_{L^2(\Omega)}:=a(w_h,v_h) \quad\forall\, w_h,v_h\in S^0_h.
		\end{align}
		By
		$P_h: L^2(\Omega)\to S_h$ and $R_h: W_0^{1,1}(\Omega) \to S^0_h$, we denote, respectively, the standard $L^2$ projection onto $S_h$ and  the Ritz projection associated with the elliptic operator $A$, i.e., they satisfy 
		$$ 
		(P_h u, v_h)_{L^2(\Omega)}=(u,v_h)_{L^2(\Omega)} \quad \forall\, u\in L^2(\Omega), v_h\in S_h, \quad
		a(R_h w,v_h)= a(w,v_h)     \quad\forall\, w\in H_0^1(\Omega) \, v_h\in S^0_h.
		$$
		For every $k=0,1$, $p\in (1,\infty]$, $m=1,2$, it is well-known that there is a constant $C>0$, depending only on $\Omega$, $\sigma$ and shape regularity constants,  such that  (see\cite[Proposition 1.134]{ern2004theory} and \cite[Theorem 2.2]{thomee2007galerkin}):
		\begin{align}
			&\|u-P_h u\|_{W^{k,p}(\Omega)}\leq Ch^{m-k}\|u\|_{W^{m,p}(\Omega)} \quad  \forall u\in  W^{m,p}(\Omega) \quad \forall h>0  \label{L2projection}\\
			& \|u-R_h u\|_{H^{k}(\Omega)}\leq Ch^{m-k}\|u\|_{H^{m}(\Omega)} \quad  \forall u\in  H^{m}(\Omega) \quad \forall h>0 \label{approx:Ritz} 
			\\
			&  \|P_h u\|_{L^{p}(\Omega)}\leq C\|u\|_{L^p(\Omega)}, \quad \|R_h u\|_{W^{1,p}(\Omega)}\leq C\|u\|_{W^{1,p}(\Omega)} \quad \forall u\in W^{1,p}(\Omega) \quad \forall h>0. \label{stability:Ritz and L2} 
   		\end{align}
     Our error analysis also relies on the use of the Carstensen quasi-interpolation operator \cite{carstensen1999quasi}, defined by 
       \begin{equation}\label{def:quasi-interp}
       \mathcal I_h:L^1(\Omega)\to S_h, \quad  \mathcal I_h(u):= \sum_{i=1}^{{N}_\mathcal{T}} \pi_i(u)\phi_i \quad 
        \pi_i(u):=\frac{\int_{\omega_i}\phi_i u \,dx}{\int_{\omega_i} \phi_i \,dx}\,\,\textup{with}
        \,\,\omega_i:=\textup{supp}\,\phi_i,
       \end{equation}
       where $\{\phi_i\}_{i=1}^{{N}_\mathcal{T}}\subset S_h$ denotes the canonical nodal basis of $S_h$. It is well-known (see \cite[Theorem 3.1]{carstensen1999quasi}) that 
        $\mathcal I_h:L^1(\Omega)\to S_h$ enjoys the   $H^1$-stability: There exists   a constant  $C>0$ independent of $h>0$ such that 
       \begin{equation}\label{Pih:grad:stability}
        \|\nabla \mathcal I_h u\|_{L^2(\Omega)}  \leq C\|\nabla u\|_{L^2(\Omega)}   \quad\forall u\in H^1(\Omega).
       \end{equation}
     By definition, it also satisfies 
       \begin{equation}\label{preserve:A}
        \underline q\leq   q(x)\leq \overline{q}  \textrm{ for  a.e. } x \in \Omega    \implies    \underline q\leq     \mathcal I_h(q)(x)  \leq \overline{q} \textrm{ for  a.e. } x \in \Omega. 
       \end{equation}
      Furthermore, according to \cite[Lemma 4.4]{los2008finite} along with  
      \begin{equation}
       \|u\|_{H^{-1}(\Omega)}= \sup_{\varphi\in H_0^1(\Omega)\setminus\{0\}} \frac{(u,\varphi)_{L^2(\Omega)}}{\|\varphi\|_{H^1(\Omega)}} \leq 
       \sup_{\varphi\in H^1(\Omega)\setminus\{0\}}\frac{(u,\varphi)_{L^2(\Omega)}}{\|\varphi\|_{H^1(\Omega)}} = \|u\|_{H^1(\Omega)^*} \quad \forall u  \in L^2(\Omega),
      \end{equation}
      there exists a constant $C>0$, independent of $h$, such that  
       \begin{equation}\label{lemma:negative-norm}
        \|\mathcal I_h(u)-u\|_{H^{-1}(\Omega)}\leq C h^2 \|u\|_{H^1(\Omega)}\quad\forall u\in H^1(\Omega).
       \end{equation} 
		Moreover, a priori error estimation for  \eqref{numerical:op} is obtained as follows:
        \begin{lemma}[{\cite[Theorem 2]{as2002uniform}}]\label{thm:fem:app}
			Let Assumption \ref{Assump: FEM} be satisfied.  
            Then,  there exist constants  $h_0,C>0$ such that  
				\begin{align}
					&\|u(q)-u_h(q)\|_{L^2(\Omega)}+  h\|u(q)-u_h(q)\|_{H^1(\Omega)}\leq C h^2 \quad \forall h\in (0,h_0) \quad\forall q\in \mathcal{A} \label{opM:error:H1} \\
                    &C_{\mathcal{A}}:=\sup_{h\in (0,h_0)} \sup_{q\in \mathcal{A}}\big(\|u_h(q)\|_{H^1(\Omega)}+\|u_h(q)\|_{C(\overline{\Omega})}\big)<\infty. \label{opM:error:H1:2}
                    \end{align}
		\end{lemma}
        
		\begin{lemma}\label{lemma:FEM:app2}
			Let Assumption \ref{Assump: FEM} hold, and $h_0>0$ be as in Lemma \ref{thm:fem:app}.  Then, there is a constant $C>0$, independent of $h$,  such that 
			\begin{equation}\label{l2:error2}
					\big\|u(q^\dag)-u_h(\mathcal I_h q^\dag)\big\|_{L^2(\Omega)}+ h \big\|\nabla (u(q^\dag)-u_h(\mathcal I_h q^\dag))\big\|_{L^{2}(\Omega)}\leq C h^2  \quad\forall h\in (0,h_0).
			\end{equation}

		\end{lemma}
		\begin{proof}
         By \eqref{opM:error:H1}, it suffices to show that there is a constant $C>0$ such that
        	\begin{equation}\label{l2:error2:0}
					\big\|u_h(q^\dag)-u_h(\mathcal I_h q^\dag)\big\|_{L^2(\Omega)}+  h\big\|\nabla (u_h(q^\dag)-u_h(\mathcal I_h q^\dag))\big\|_{L^{2}(\Omega)}\leq C h^2 \quad \forall h\in (0,h_0).
			\end{equation}
         Thanks to Assumption \ref{Assump: FEM} and \eqref{preserve:A}, it holds that  $\mathcal I_h q^\dag\in \mathcal{A}$ for all $h>0$. 
         Let us next prove that
          \begin{equation}\label{estimate:W1infty}
           C_{\infty}:= \sup_{h\in (0,h_0)}\|u_h(q^\dag) \|_{W^{1,\infty}(\Omega)}<\infty.  
          \end{equation}
          Using the $L^2$ projection operator $P_h$ and the Ritz projection operator $R_h$,  it follows from \eqref{numerical:op} that
			\begin{align}\label{lemma:lp regular eq2}
				A_h u_h(q^\dag) + P_h q^\dag  u_h(q^\dag) = P_h  f & \quad \ \,  \implies \qquad    u_h(q^\dag) = A_h^{-1} P_h ( f- q^\dag  u_h(q^\dag)^m  ) \\ 
               & \underbrace{\implies}_{A_h^{-1}P_h =  R_h A^{-1}} \    u_h(q^\dag)= R_h A^{-1}   ( f- q^\dag  u_h(q^\dag)^m   ). 
			\end{align}
              On the other hand, we know that  $A^{-1} \in \mathcal{L}(L^r(\Omega), W^{1,\infty}(\Omega))$  due to the convexity assumption on $\Omega$ and the regularity $\sigma\in C^{0,1}(\overline \Omega)^{3\times3}$ (see \eqref{est:grad}). Thus, along with  H\"{o}lder's inequality and \eqref{stability:Ritz and L2},   it follows that 
\begin{align*}
\|u_h(q^\dag)\|_{W^{1,\infty}(\Omega)} &\,\,= \|R_h A^{-1}   (f-q^\dag  u_h(q^\dag)^m )\|_{W^{1,\infty}(\Omega)}  \\
&\!\underbrace{\le}_{\eqref{stability:Ritz and L2}} \| A^{-1}   ( f-  q^\dag  u_h(q^\dag)^m )\|_{W^{1,\infty}(\Omega)}  
\le C \| f- q^\dag  u_h(q^\dag)^m\|_{L^{r}(\Omega)} \\
&\,\,\!\leq C (\| f\|_{L^r(\Omega)} +\| q^\dag\|_{L^\infty(\Omega)}\|u_h(q^\dag)\|^m_{C(\overline \Omega)}) \\
&\!\!\underbrace{\leq }_{\eqref{opM:error:H1:2}} C(\| f\|_{L^r(\Omega)} + \overline{q} C_{\mathcal{A}}^m) \quad \forall h \in (0,h_0).  
\end{align*}
In conclusion, \eqref{estimate:W1infty} is valid. 

With this result at hand, we now derive \eqref{l2:error2:0}. Subtracting the weak formulations  for $u_h(q^\dag)$ and $u_h(\mathcal I_h  q^\dag)$ (cf.  \eqref{numerical:op}) from each other and then using the algebraic identity \eqref{difference binomial}, we deduce with $  W_h:=\sum_{k=0}^{m-1} u_h(\mathcal I_h  q^\dag)^k u_h(q^\dag)^{m-1-k}$ that  
			\begin{align}\label{ineq1}
				\begin{split}
              &\quad \, a(u_h(\mathcal I_h  q^\dag)-u_h(q^\dag),v_{h}) +  (\mathcal I_h  q^\dag  W_h  (u_h(\mathcal I_h  q^\dag)-u_h(q^\dag)), v_{h})_{L^2(\Omega)} \\
            & =  ((q^\dag-\mathcal I_h  q^\dag) u_h(q^\dag)^m,v_{h})_{L^2(\Omega)} \quad\forall v_h\in S_h^0. 
				\end{split}
			\end{align}
			On the other hand, thanks to \eqref{difference binomial ineq} and the non-negativity of $\mathcal I_h q^\dag \in \mathcal A$, we have that   $\mathcal I_h  q^\dag  W_h \geq 0$. Thus, it follows from \eqref{ineq1} and Poincar\'e inequality  that 
        \begin{equation}\label{eh:h1}
          \|u_h(\mathcal I_h  q^\dag)-u_h(q^\dag)\|_{H^1(\Omega)}\leq C \|(\mathcal I_h  q^{\dag} - q^\dag) u_h(q^\dag)^m\|_{H^{-1}(\Omega)} \quad \forall h\in(0, h_0) 
        \end{equation}
        with $C>0$ depending only on $\Omega$. By the definition of $H^{-1}(\Omega)$, it follows that 
        \begin{align*}
        \|(\mathcal I_h  q^{\dag} - q^\dag) u_h(q^\dag)^m\|_{H^{-1}(\Omega)}&:= \sup_{\|\varphi\|_{H_0^1(\Omega)}=1}\{ ((\mathcal I_h  q^\dag - q^\dag) u_h(q^\dag)^m, \varphi)_{L^2(\Omega)}\}  \\
        &\ \leq  \|\mathcal I_h  q^\dag -q^\dag\|_{H^{-1}(\Omega)}  \sup_{\|\varphi\|_{H_0^1(\Omega)}=1}\{\|u_h(q^\dag)^m\varphi\|_{H_0^1(\Omega)}\} \quad \forall h\in(0,h_0). 
        \end{align*}
        Combined with \eqref{lemma:negative-norm} and the chain rule, this finally leads to 
        \begin{equation}\label{proof lem ineq}
          \|(\mathcal I_h q^{\dag} - q^\dag) u_h(q^\dag)^m\|_{H^{-1}(\Omega)}\leq C h^2  \|u(q_h)^m\|_{W^{1,\infty}(\Omega)}
          \leq C C_\infty^m h^2 \quad \forall h  \in (0,h_0). 
        \end{equation}
         Applying \eqref{eh:h1} to \eqref{proof lem ineq}     implies finally \eqref{l2:error2:0}. This completes the proof.            
		\end{proof}

	    \begin{remark}\label{remark:conti:uh}
       By the same reasoning leading to \eqref{eh:h1}, we obtain   for every $h>0$   and  $q_h,q'_h\in \mathcal{A}_h$ 
       \begin{equation}
          \|u_h(q_h)-u_h(q'_h)\|_{H^1(\Omega)}\leq C \|(q_h - q'_h) u_h(q_h)^m\|_{H^{-1}(\Omega)}
          \leq  C \|q_h - q'_h\|_{L^2(\Omega)}  \|u_h(q_h)\|^m_{L^\infty(\Omega)}
       \end{equation}
       with a constant $C>0$   independent of $h>0$.   For every fixed $h>0$, this estimate implies in particular the continuity of the mapping $u_h:\mathcal{A}_h\subset S_h\to S_h$.

	    \end{remark}

		\begin{theorem} \label{final theo}
			Let Assumptions  \ref{Assump: cse} and \ref{Assump: FEM} be satisfied, and $h_0>0$  be as in Lemma \ref{thm:fem:app}. 
			Then, there is a constant $C>0$, independent of $h,\delta,$ and $\alpha$, such that 
			\begin{align}\label{discrete u:1p}
               \qquad  \| u^\dag- u_h(\qda)\|_{L^2(\Omega)}&\leq C\overbrace{(h^2+\delta+\sqrt{\alpha})}^{=:\eta}
            \\
            \label{discrete:1p}
				\|q^\dag -\qda\|_{L^2(\Omega)}&\leq C ( \eta+h+\min\left\{ h+h^{-1}(\delta + \sqrt{\alpha}) ,1\right\})^{\frac{1}{1+(m+1)\gamma}} (1+\alpha^{-\frac{1}{2}} \eta)^\frac{(m+1)\gamma}{1+(m+1)\gamma} 
			\end{align}  
			hold  for every  $h \in (0,h_0)$, $\alpha, \delta>0$, and every minimizer $\qda$ of \eqref{tik dis}.

		\end{theorem}

		\begin{proof} 
			The proof is divided into the   following steps:\\
			
			\noindent (Step 1.) Let us prove the existence of a constant $C>0$,  independent of $h,\delta,$ and $\alpha$, such that
            \begin{equation}\label{first:error:coro1}
				\|u_h(\qda)-P_hu(q^\dag)\|_{L^2(\Omega)}\leq C\eta = C(h^2+\delta+\sqrt{\alpha})  \quad \forall h \in(0,h_0) \  \forall \alpha,\delta>0.
			\end{equation}
            To this aim, let  $h \in(0,h_0)$ and $ \alpha,\delta>0$ be arbitrarily fixed. Since $\qda \in \mathcal A_h$ is a minimzer of \eqref{tik dis}, we obtain from  \eqref{preserve:A}   that 
			\begin{equation}\label{mini0}
				\begin{split}
				    &\quad \, \|u_h(q^{\delta}_{h,\alpha})-y^\delta\|_{L^2(\Omega)}^2+ \alpha (\|\qda\|^2_{L^2(\Omega)}+\|\nabla \qda\|_{L^2(\Omega)}^2)= \mathcal{J}^\delta_{\alpha,h}(q^{\delta}_{h,\alpha}) \\
                    &\le \mathcal{J}^\delta_{\alpha,h}(\mathcal I_h q^\dag) =  \|u_h(\mathcal I_h q^\dag)-y^\delta\|_{L^2(\Omega)}^2+ \alpha(\|\mathcal I_h q^\dag\|^2_{L^2(\Omega)} +\|\nabla \mathcal I_h q^\dag\|_{L^2(\Omega)}^2). 
				\end{split}
			\end{equation}
			By the triangle inequality,  \eqref{noisy},\eqref{Pih:grad:stability}, \eqref{preserve:A}, and \eqref{l2:error2},  the above right hand side can be estimated by
			\begin{equation}\label{first:error1.5}
				\begin{split}
					&\quad \,  \big\|u_h(\mathcal I_h q^\dag)-y^\delta\big\|_{L^2(\Omega)}^2+ \alpha(\|\mathcal I_h q^\dag\|^2_{L^2(\Omega)}+  \big\|\nabla \mathcal I_h q^\dag\big\|_{L^2(\Omega)}^2)\\
					&\leq 2 \big\|u_h(\mathcal I_h q^\dag)-u(q^\dag)\big\|_{L^2(\Omega)}^2+2 \big\| u(q^\dag)-y^\delta \big\|_{L^2(\Omega)}^2
					+\alpha(\|\mathcal I_h q^\dag\|^2_{L^2(\Omega)} +\big\|\nabla \mathcal I_h q^\dag \big\|_{L^2(\Omega)}^2)\\
		&\!\!\!\!\!\!\!\!\underbrace{\leq}_{\eqref{noisy},\eqref{l2:error2}} \!\!\!\! C(h^{4}+ \delta^{2})+ \alpha(\|\mathcal I_h q^\dag\|^2_{L^2(\Omega)} +\big\|\nabla \mathcal I_h q^\dag\big\|_{L^2(\Omega)}^2) \underbrace{\leq}_{\eqref{Pih:grad:stability},\eqref{preserve:A}} C(h^{4}+ \delta^{2} + \alpha )  \le C \eta^2
				\end{split}
			\end{equation} 
			with a constant $C>0$,  independent of $h,\delta,$ and $\alpha$. Applying  the above estimate to \eqref{mini0} implies
			\begin{equation}\label{first:error2}
				\|u_h(q^{\delta}_{h,\alpha})-y^\delta\|_{L^2(\Omega)}^2+ \alpha (\|\qda\|^2_{L^2(\Omega)}+\|\nabla \qda\|_{L^2(\Omega)}^2)  \leq   C\eta^2. 
			\end{equation}
			 Now, using the triangle inequality, the desired estimate \eqref{first:error:coro1} is obtained as follows:  
			\begin{align}
				&  \|u_h(\qda)-P_hu(q^\dag)\|_{L^2(\Omega)}\leq  \|u_h(\qda)-y^\delta\|_{L^2(\Omega)} +\|y^\delta- u(q^\dag)\|_{L^2(\Omega)} + 
				\|u(q^\dag)-P_hu(q^\dag)\|_{L^2(\Omega)}\\
				&  \underbrace{\le}_{\eqref{noisy},\eqref{L2projection},\eqref{first:error2}} C\eta+ \delta+ h^2 \|u(q^\dag)\|_{H^{2}(\Omega)} \le (C+1+ \|u(q^\dag)\|_{H^{2}(\Omega)})\eta. 
			\end{align}
As a consequence of \eqref{first:error:coro1}, $u^\dag=u(q^\dag) \in H^2(\Omega)$, and \eqref{L2projection}, we conclude that \eqref{discrete u:1p} is valid.\\
			
            \noindent {(Step 2.)}  Let us prove the existence of a constant $C>0$,  independent of $h,\delta,$ and $\alpha$, such that
			\begin{equation}\label{discrete:w1p}
				\|q^\dag -\qda \|_{{\mathcal{X}^{2}_{u^\dag}(\Omega)^*} }\leq C ( \eta+h+\min\left\{ h+ h^{-1}(\delta + \sqrt{\alpha}) ,1\right\}).
			\end{equation}
			To this end, let   $h \in(0,h_0)$ and $ \alpha,\delta>0$ be arbitrarily fixed.  We notice  that   the variational formulations of $u(q^\dag)$  and   $u_h(\qda)$
            $$ 
			\begin{aligned}
            a(u(q^\dag), v)+(q^\dag u(q^\dag)^m, v)_{L^2(\Omega)}&=( f, v)_{L^2(\Omega)} \quad\forall v\in H^1_0(\Omega)\\
			    	a(u_h(\qda), v_h)+(\qda u_h(\qda)^m, v_h)_{L^2(\Omega)}&=( f, v_h)_{L^2(\Omega)} \quad\forall v_h\in S_{h}^0
			\end{aligned}
            $$
			imply
			\begin{equation}\label{discre:cse:1}
				\begin{split}
					 \int_{\Omega} (q^\dag-\qda) u(q^\dag)^m v_h\,  dx 
					= a(u_h(\qda)-u(q^\dag), v_h)  +(\qda (u_h(\qda)^m-u(q^\dag)^m),v_h)_{L^2(\Omega)} \quad \forall v_h\in S_h^0. 
				\end{split}
			\end{equation}
			Now, let $v\in H^1_0(\Omega)$ be arbitrarily fixed. 		Then, applying the identity 
			\begin{align}
				\int_{\Omega} (q^\dag-\qda) u(q^\dag)^m v\, dx
				=  \int_{\Omega} (q^\dag-\qda) u(q^\dag)^m P_h v \, dx  +\int_{\Omega} (q^\dag-\qda) u(q^\dag)^m (I-P_h)v \, dx
			\end{align}
			to \eqref{discre:cse:1} with $v_h = P_h v$ yields that 
			\begin{equation}\label{discre:cse:2}
				\begin{split}
					\int_{\Omega} (q^\dag-\qda) u(q^\dag)^m v \,  dx
					&=a(u_h(\qda)-u(q^\dag), P_h v) +(\qda (u_h(\qda)^m-u(q^\dag)^m),P_h v)\\
					& \quad\, + \int_{\Omega} (q^\dag-\qda) u(q^\dag)^m (I-P_h)v\,  dx=: {\bf J_1}+{\bf J_2}+{\bf J_3}. 
				\end{split}
			\end{equation}
			In the sequel, we estimate these three terms separately. To estimate ${\bf J_1}$, we first deduce from  \eqref{L2projection} and the inverse estimate  \eqref{inverse:property}   that  
			\begin{equation}\label{discre:cse:3}
				\begin{split}
					{\bf J_1} 
					&\leq \|\nabla (u_h(\qda)-u(q^\dag))\|_{L^2(\Omega)}\|\nabla P_h v\|_{L^{2}(\Omega)} \\ 
                    &\leq C  \big( \|\nabla (u_h(\qda)-P_h u(q^\dag))\|_{L^2(\Omega)}
					  +\|\nabla (P_h u(q^\dag)-u(q^\dag))\|_{L^2(\Omega)}\big)\|v\|_{H^{1}(\Omega)}\\
					&\leq C  \big( h^{-1}\|(u_h(\qda)-P_h u(q^\dag))\|_{L^2(\Omega)}+h\|u(q^\dag)\|_{H^{2}(\Omega)}\big)\|v\|_{H^{1}(\Omega)}  \\
                    &\!\!\!\underbrace{\leq}_{\eqref{first:error:coro1}} C\left(h+ h^{-1}(\delta + \sqrt{\alpha})\right)\|v\|_{H^{1}(\Omega)}.   
				\end{split}
			\end{equation}
			On the other hand, Lemma \ref{thm:fem:app} along with $\qda\in \mathcal{A}$ yields that $\|u_h(\qda)\|_{H^1(\Omega)}\leq C_{\mathcal{A}}$,  and consequently
			$$
			\|\nabla(u_h(q)-u(q^\dag))\|_{L^2(\Omega)}\leq   C_{\mathcal{A}} +\|u(q^\dag)\|_{H^{1}(\Omega)}.
			$$
			Combined with \eqref{discre:cse:3}, this estimate shows that 
			\begin{equation}\label{discre:cse:J}
				{\bf J_1} \leq C\min\left\{ h+ h^{-1}(\delta + \sqrt{\alpha}) ,1\right\}\|v\|_{H^{1}(\Omega)}. 
			\end{equation}
			For the second term  ${\bf J_2}$, we deduce that
			\begin{equation}\label{discre:cse:5}
				\begin{split}
					{\bf J_2}&\leq C  \| u_h(\qda)^m-u(q^\dag)^m\|_{L^2(\Omega)}\|P_h v\|_{L^{2}(\Omega)} \\
                    &\!\! \underbrace{\leq}_{\eqref{difference binomial}} \|\sum_{k=0}^{m-1}u_h(\qda)^{m-1-k} u(q^\dag)^k \|_{L^\infty(\Omega)}  \| u_h(\qda)-u(q^\dag)\|_{L^2(\Omega)}\|P_h v\|_{L^{2}(\Omega)}\\
				&\!\!\!\!\!\!\!\underbrace{\leq}_{\eqref{opM:error:H1:2},\eqref{stability:Ritz and L2}} C (m+1) C_{\mathcal{A}}^{m-1}\|u_h(\qda)-u(q^\dag)\|_{L^2(\Omega)} \|v\|_{L^{2}(\Omega)}  \underbrace{\leq}_{\eqref{first:error:coro1}} C \eta \|v\|_{H^{1}(\Omega)}.
				\end{split}
			\end{equation}
			To estimate the final term ${\bf J_3}$, we utilize    \eqref{L2projection}  to obtain 
			\begin{equation}\label{discre:cse:6}
				{\bf J_3}\leq \| (q^\dag-\qda) u(q^\dag)^m \|_{L^2(\Omega)} \|(I-P_h) v\|_{L^{2}(\Omega)}\leq 2 \overline{q}\|u(q^\dag)\|^m_{C(\overline \Omega)}  C  h \|v\|_{H^{1}(\Omega)}.
			\end{equation} 
			Since $v\in H^1_0(\Omega)$  was chosen arbitrarily,  it follows from \eqref{discre:cse:2}, \eqref{discre:cse:J}, \eqref{discre:cse:5} and \eqref{discre:cse:6} that 
			\begin{equation}\label{discre:cse:7}
				\int_{\Omega} (q^\dag-\qda) u(q^\dag)^m v \,  dx\leq C \left( \eta+h+\min\left\{ h+ h^{-1}(\delta + \sqrt{\alpha}) ,1\right\}\right)\|v\|_{H^{1}(\Omega)} \quad\forall v\in  H^{1}_0(\Omega).
			\end{equation}
			Applying the same argument as in Theorem \ref{thm: cse} to \eqref{discre:cse:7} gives 
			the desired estimate \eqref{discrete:w1p}. \\
            
			\noindent {(Step 3.)} From \eqref{first:error2}, it follows that 
			\begin{equation} \label{final theo0}
				\big\|q^\dag-\qda \big\|_{H^{1}(\Omega)}\leq  \|q^\dag\|_{H^{1}(\Omega)}+\|\qda \big\|_{H^{1}(\Omega)}\leq   \|q^\dag\|_{H^{1}(\Omega)} +C \alpha^{-\frac{1}{2}} \eta. 
			\end{equation}
            Finally, invoking the interpolation inequality \eqref{Gamma:interp2} of Lemma \ref{lemma:W dual} with $p=2$, $\overline p= \infty$, and $\kappa= m\gamma+\gamma$ yields
\[\begin{split}
            \|q^\dag-\qda\|_{L^{2}(\Omega)}&\leq
				C\|q^\dag-\qda\|_{\mathcal{X}^{2}_{u^\dag}(\Omega)^*}^{\frac{1}{1+\kappa}}\|q^\dag-\qda\|_{H^1(\Omega)}^\frac{\kappa}{1+\kappa} \\
                &\le C ( \eta+h+\min\left\{ h+ h^{-1}(\delta + \sqrt{\alpha}) ,1\right\})^{\frac{1}{1+\kappa}} (1+\alpha^{-\frac{1}{2}} \eta)^\frac{\kappa}{1+\kappa}\\
                &= C ( \eta+h+\min\left\{ h+h^{-1}(\delta + \sqrt{\alpha}) ,1\right\})^{\frac{1}{1+(m+1)\gamma}} (1+\alpha^{-\frac{1}{2}} \eta)^\frac{(m+1)\gamma}{1+(m+1)\gamma}.
                \end{split}
 \]          
Thus, the claim is valid. 
			\end{proof}
		{
		\begin{remark}\label{remark:estimator}
			Under the settings  $h\sim \sqrt{\delta}$ and $\sqrt{\alpha} \sim \delta$, Theorem \ref{final theo} implies the following error estimates
			$$
             \|u^\dag-u(\qda)\|_{L^2(\Omega)}\leq C \delta \quad \textrm{and} \quad 
             \quad  \|q^\dag -\qda\|_{L^{2}(\Omega)}\leq C \delta^{\frac{1}{2(1+(m+1)\gamma)}}.
            $$

\end{remark}
            }
		\begin{remark}\label{remark:comp:FEM}
        Theorem \ref{final theo} significantly improves the existing error estimate by Jin et al. \cite{jin2022convergence} for the linear case $m=1$ with regards to two aspects. First, as pointed out in Remark \ref{H1 ass only}, the error estimation \eqref{discrete:1p} is readily obtained under the $H^1(\Omega)$-regularity assumption on the true coefficient, omitting the $H^2(\Omega)$-regularity assumption from the prior contribution. Secondly, the achieved error estimate \eqref{discrete:1p} improves the convergence order \cite[Corollary 3.2]{jin2022convergence} to the power of two. This improved convergence rate was readily observed in the numerical tests by \cite{jin2022convergence}. Our numerical tests (see below) confirm our theoretical finding. 
        \end{remark}

		\section{Numerical simulation and discussions}\label{sec:numerics} 
		In this section, we present two numerical tests to validate our theoretical analysis.  All implementations were done in Python using the open-source FEniCS \cite{logg2010dolfin} on a Mac mini (Apple M4 Pro chip, 24 GB RAM) running macOS.  Furthermore, the discrete minimization problem \eqref{tik dis} is solved by the BFGS-B algorithm available through the \texttt{scipy.optimize} package in Python with the initial guess $q_0\equiv 1.0$. Regarding the data for our tests, we set $\sigma \equiv 1$, $\underline q= 0$, $\overline q =2$, $m=1$, and the noisy data $y^\delta$ is generated by 
		$$
		y^\delta = u^\dag + \varepsilon \|u^\dag\|_{L^\infty(\Omega)}\xi
		$$
		with $\xi$ being the standard uniform distribution. Two scenarios for the true coefficient $q^\dag$ and the corresponding true solution $u^\dag$ are considered as follows:

		\begin{example}\label{numer:example}
			~\begin{enumerate}
				
				\item[\textup{(a)}] $\Omega :=(0,1)$, $u^\dag = \sin (\pi x)$, $q^\dag(x)=1.5-|x-0.5|$. 
				
				\item[\textup{(b)}] $\Omega :=(0,1)^2$, $u^\dag = \sin(\pi x_1)\sin(\pi x_2)$, $q^\dag(x):=1+ 0.5(1 - 2\max\{|x_1- 0.5|, |x_2 -0.5|\})$.
				
			\end{enumerate}
			
		\end{example}
        \noindent
 Note that $q^\dag\in H^1(\Omega)\backslash H^2(\Omega)$ holds in both cases of Example \ref{numer:example}. The discretization for Example \ref{numer:example} uses a uniform mesh of mesh size $h$. Moreover, the noisy level and the regularization parameter are set to be 
\[
\begin{cases}
\delta = h^2 \quad \text{and} \quad \alpha =10^{-2}\times \delta^2 & \text{for case (a)}, \\
\delta = h^2 \quad \text{and} \quad \alpha = 5 \times\delta^2 & \text{for case (b)}.
\end{cases}
\]
This choice corresponds to Remark \ref{remark:estimator}.
 

	 Tables \ref{tab:a}-\ref{tab:b} depict the computed numerical error
		$$
        e_q(\delta):=\|u^\dag - u_h(\qda)\|_{L^2(\Omega)} \quad \textrm{and} \quad  e_u(\delta):=\|q^\dag - \qda\|_{L^2(\Omega)}
        $$
        as well as the corresponding experimental order of convergence 
\begin{equation}
 \textup{ECO}_{u}:=\frac{\log(e_u(\delta_1))-\log(e_u(\delta_1))}{\log\delta_1 -\log\delta_2}\quad\textup{and}\quad
 \textup{ECO}_{q}:=\frac{\log(e_q(\delta_1))-\log(e_q(\delta_1))}{\log\delta_1 -\log\delta_2}\ 
\end{equation}
for two consecutive noise level $\delta_1$ and $\delta_2$. 
        As pointed out in Remark \ref{remark:estimator}, our theoretical finding predicts $\textup{ECO}_u=1$ at best and $\textup{ECO}_q=0.5$ at best, respectively. The numerical results for the one-dimensional and two-dimensional cases of Example \ref{numer:example}  confirm these predictions. A clear convergence of the errors $e_u(\delta)$ and $e_q(\delta)$ is observed as $\delta \to 0$, with the experimental order of convergence closely matching the theoretical ones.  
		Finally,  the reconstructed coefficient $\qda$ is visually demonstrated in the plots (Figs. \ref{fig:three_cases} and  \ref{fig:q_comparison}), which illustrate its convergence to $q^\dag$ as  $h \to  0$.
\begin{table}[htbp]
    \centering
     \caption{Examples \ref{numer:example} (a): Convergence behavior  of numerical solutions   }
    \begin{tabular}{|c|c|c|c|c|c|}
        \hline
        $h$ & $\delta$ & $e_u(\delta)$ & $\textup{EOC}_u$ & $e_q(\delta)$ & $\textup{EOC}_q$ \\
        \hline
        $1.56\times10^{-2}$ & $2.44\times10^{-4}$ & $9.34\times10^{-5}$ & -- & $7.12\times10^{-2}$ & -- \\
        $7.81\times10^{-3}$ & $6.1\times10^{-5}$ & $2.59\times10^{-5}$ & 0.925 & $3.28\times10^{-2}$ & 0.560 \\
        $3.91\times10^{-3}$ & $1.53\times10^{-5}$ & $6.2\times10^{-6}$ & 1.031 & $1.31\times10^{-2}$ & 0.664 \\
        $1.95\times10^{-3}$ & $3.81\times10^{-6}$ & $1.55\times10^{-6}$ & 1.002 & $6.26\times10^{-3}$ & 0.530 \\
        \hline
    \end{tabular}
    \label{tab:a}
\end{table}

\begin{table}[htbp]
    \centering
     \caption{Examples \ref{numer:example} (b): Convergence behavior of numerical solutions}
    \begin{tabular}{|c|c|c|c|c|c|}
        \hline
        $h$ & $\delta$ & $e_u(\delta)$ & $\textup{EOC}_u$ & $e_q(\delta)$ & $\textup{EOC}_q$ \\
        \hline
        $1.41\times10^{-1}$ & $2.00\times10^{-2}$ & $3.44\times10^{-3}$ & -- & $2.41\times10^{-1}$ & -- \\
        $6.43\times10^{-2}$ & $4.13\times10^{-3}$ & $9.42\times10^{-4}$ & 0.822 & $1.14\times10^{-1}$ & 0.474 \\
        $4.16\times10^{-2}$ & $1.73\times10^{-3}$ & $3.75\times10^{-4}$ & 1.059 & $7.14\times10^{-2}$ & 0.537 \\
        $2.83\times10^{-2}$ & $8.00\times10^{-4}$ & $1.77\times10^{-4}$ & 0.972 & $5.16\times10^{-2}$ & 0.420 \\
        \hline
    \end{tabular}
    \label{tab:b}
\end{table}

		\begin{figure}[htbp]
			\centering
            	\caption{Exact and recovered coefficients for Example \ref{numer:example} (a) }
			\begin{subfigure}{0.32\textwidth}
				\centering
				\includegraphics[width=\linewidth]{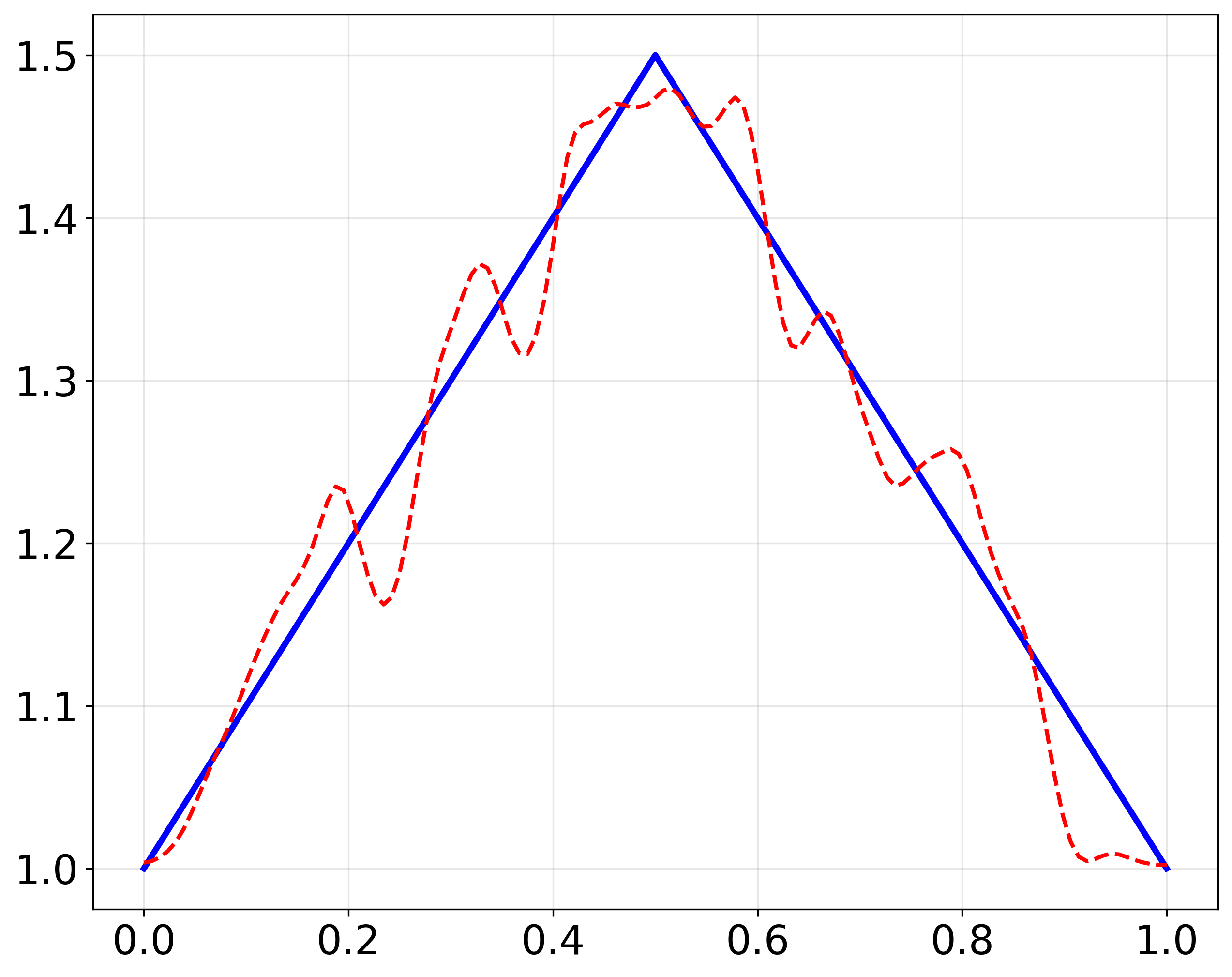}				\caption*{$h=7.81\times10^{-3}\,(N=128)$}
				\label{fig:case_80}
			\end{subfigure}
			\hfill
			\begin{subfigure}{0.32\textwidth}
				\centering
				\includegraphics[width=\linewidth]{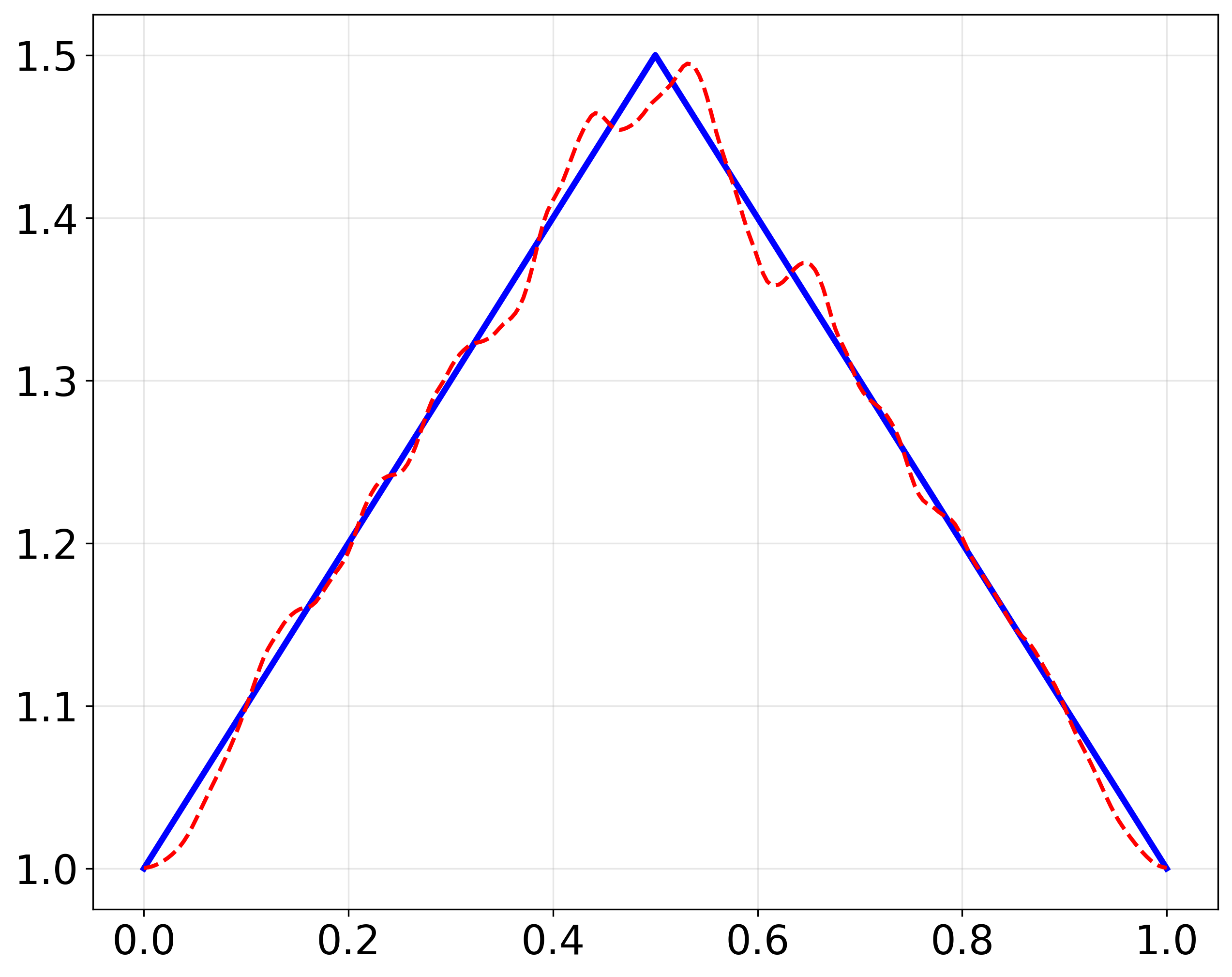}
				\caption*{$h=3.91\times10^{-3}\,(N=256)$}
				\label{fig:case_160}
			\end{subfigure}
			\hfill
			\begin{subfigure}{0.32\textwidth}
				\centering
				\includegraphics[width=\linewidth]{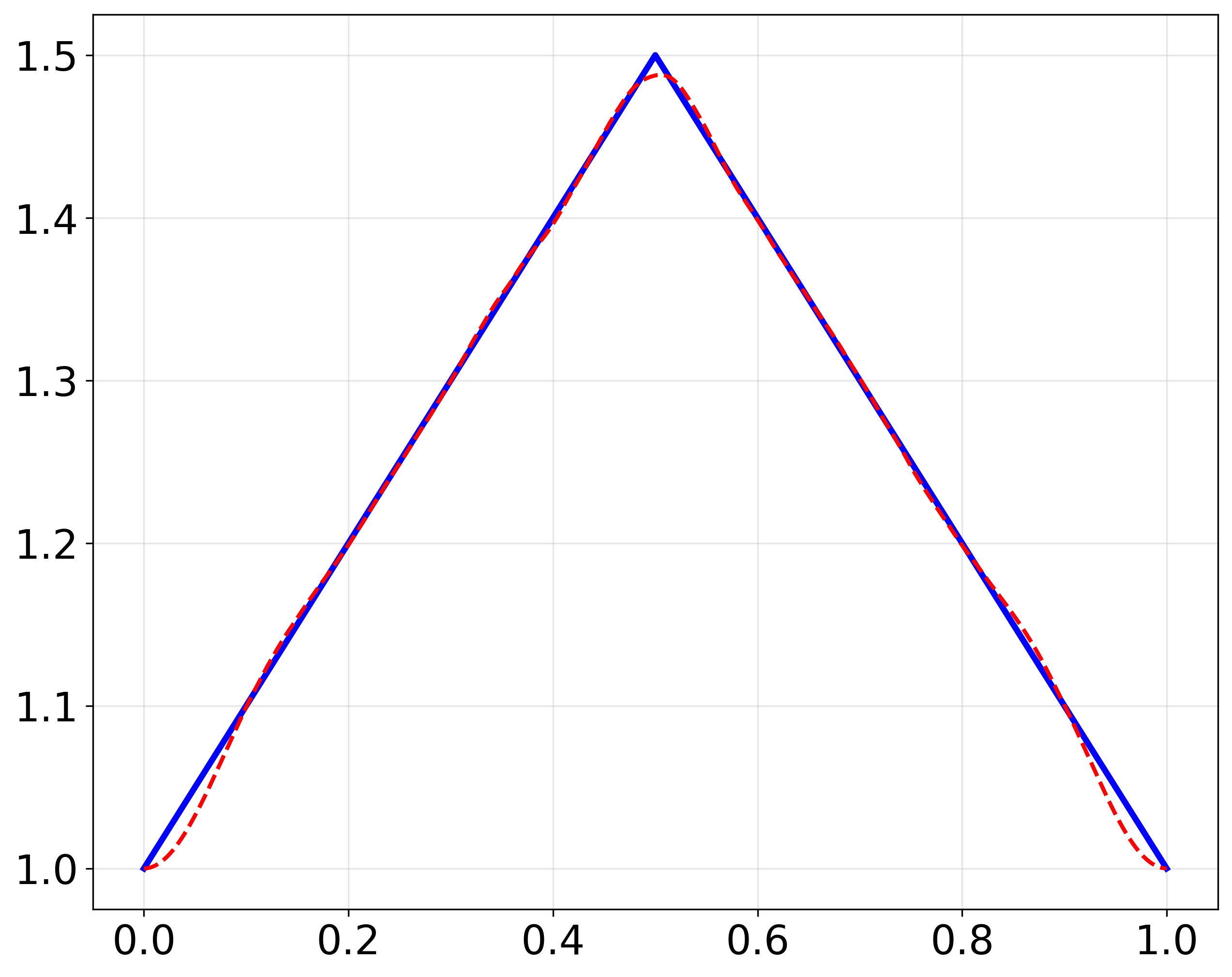}
				\caption*{$h=1.95\times10^{-3}\,(N=512)$}
				\label{fig:case_400}
			\end{subfigure}
		
			\label{fig:three_cases}
		\end{figure}

	\begin{figure}[htbp]\label{fig1}
    \centering
     \caption{Exact and recovered coefficients for Example \ref{numer:example} (b)}
    \begin{subfigure}{0.32\textwidth}
        \centering
        \includegraphics[width=\textwidth]{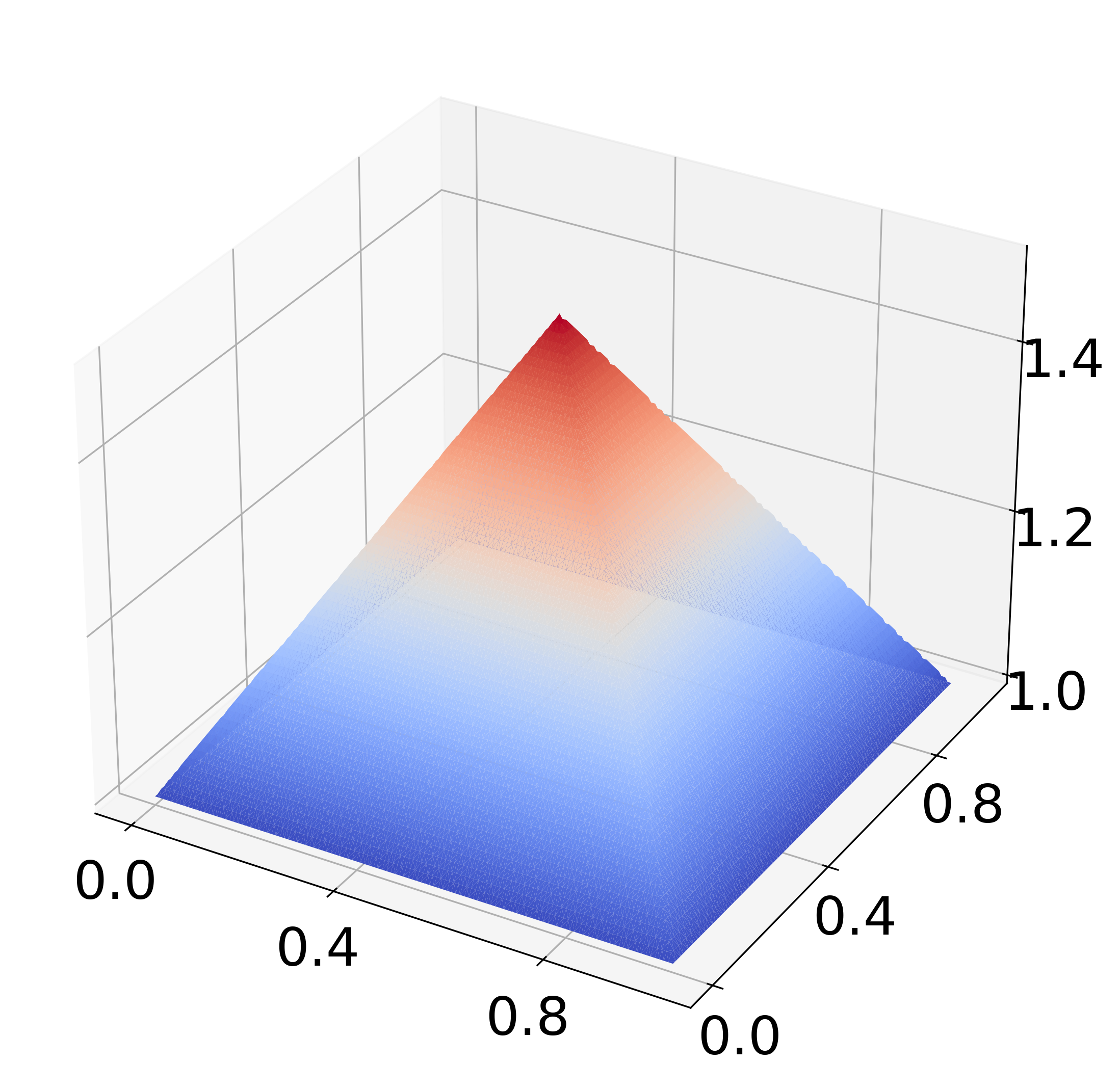}
        \caption*{{Exact coefficient $q^\dag$}}
        \label{fig:q_exact}
    \end{subfigure}
    \hfill
    \begin{subfigure}{0.32\textwidth}
        \centering
        \includegraphics[width=\textwidth]{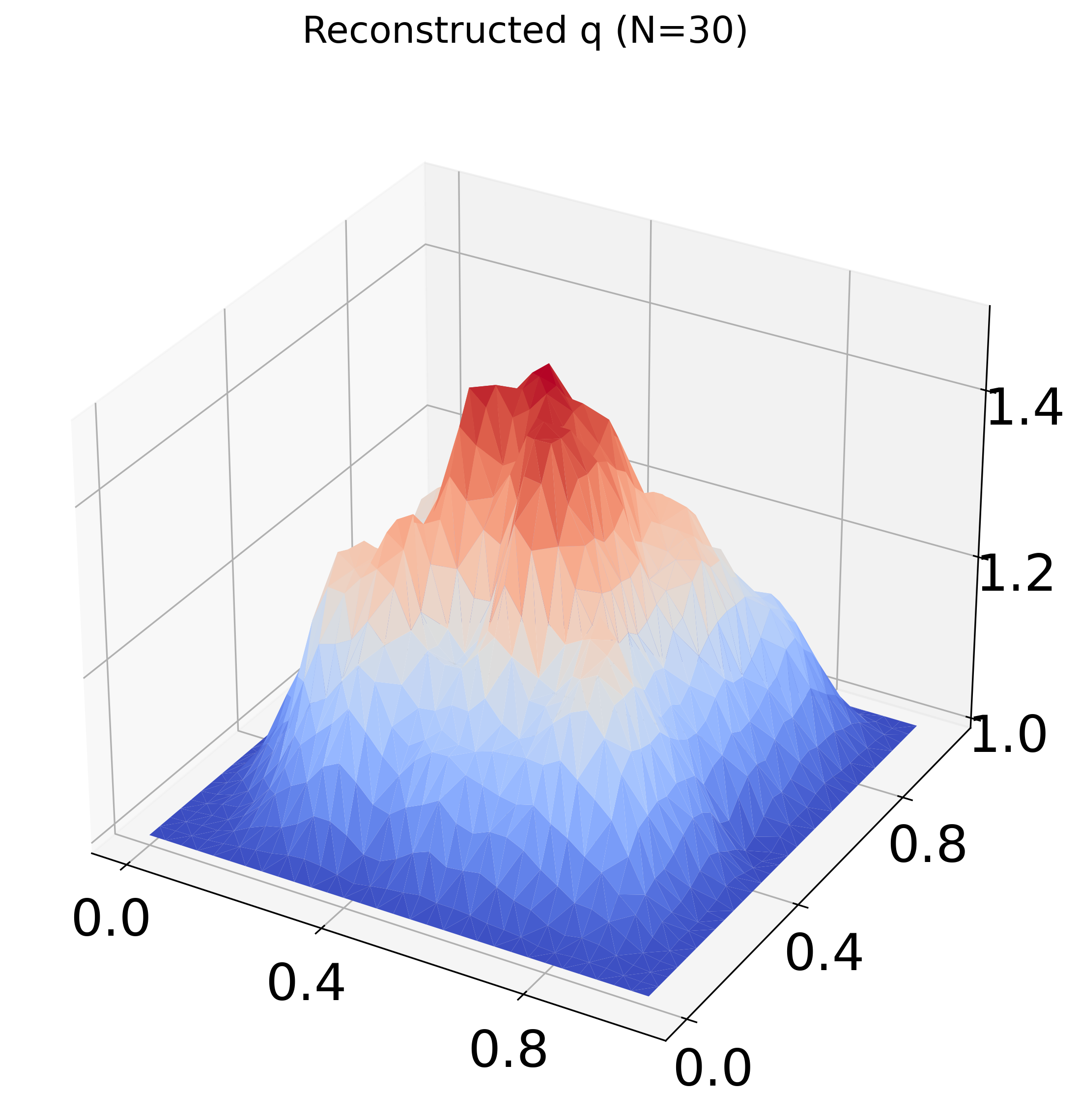}
        \caption*{Recovered coefficient  $h=4.16\times 10^{-2}$ ($N=34$)}
        \label{fig:q_inv_20}
    \end{subfigure}
    \hfill
    \begin{subfigure}{0.32\textwidth}
        \centering
        \includegraphics[width=\textwidth]{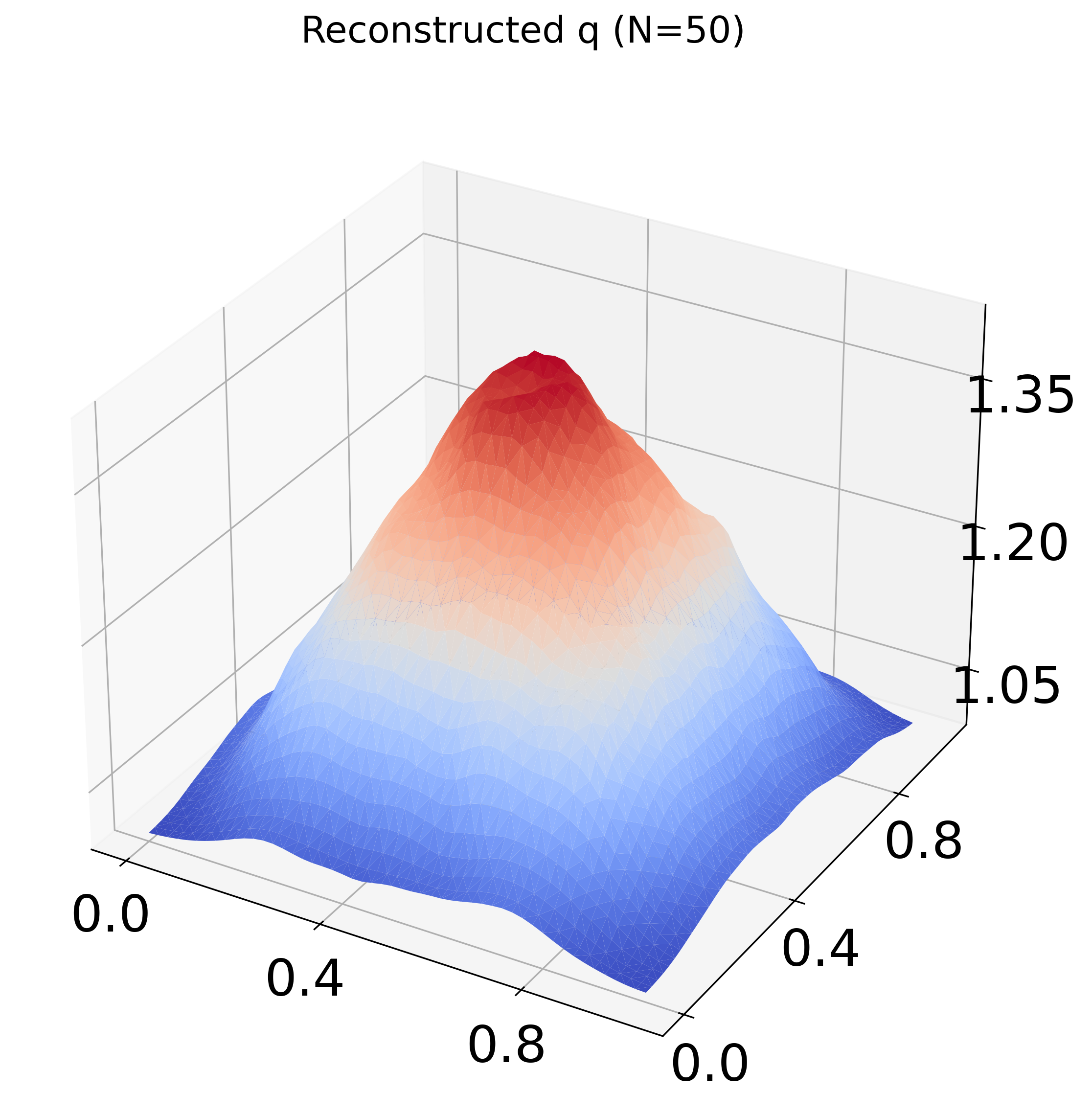}
        \caption*{Recovered coefficient   $h=2.83\times 10^{-2}$ ($N=50$)}
        \label{fig:q_inv_30}
    \end{subfigure}
    \label{fig:q_comparison}
\end{figure}

\newpage

		\noindent\textbf{Acknowledgments.} The work of De-Han Chen and Irwin Yousept are supported by  DFG research grant YO 159/5-1, project number: 513566305.
		Yi-Hsuan Lin is partially supported by the National Science and Technology Council (NSTC), Taiwan, under the project 113-2115-M-A49-017-MY3. Y.-H. L. acknowledges financial supports from the Alexander von Humboldt Foundation through the Henriette Herz Scouting Programme and hosted by Universit\"at Duisburg-Essen		

		\bibliography{refs} 
		
		\bibliographystyle{siam}
		
	\end{document}